\documentclass[10pt]{amsart}
\usepackage{graphicx, xcolor}
\usepackage{amssymb}
\numberwithin{equation}{section}

\def\N{\mathbb{N}}
\def\R{\mathbb{R}}

\renewcommand\d{\partial}

\def\k{\kappa}

\def\epsilon{\varepsilon}
\def\e{\varepsilon}

\newcommand\br{\begin{rem}}
\newcommand\er{\end{rem}}
\newcommand\bp{\begin{pmatrix}}
\newcommand\ep{\end{pmatrix}}
\newcommand\be{\begin{equation}}
\newcommand\ee{\end{equation}}
\newcommand\ba{\begin{equation}\begin{aligned}}
\newcommand\ea{\end{aligned}\end{equation}}

\newtheorem{theorem}{Theorem}[section]
\newtheorem{proposition}[theorem]{Proposition}
\newtheorem{lemma}[theorem]{Lemma}

\newtheorem{example}[theorem]{Example}
\newtheorem{remark}[theorem]{Remark}

\title{Stability  properties and dynamics of solutions to  viscous conservation laws  with mean curvature operator} 

\begin{document}

\maketitle

\begin{center}
RAFFAELE FOLINO\footnote{Universit\` a degli Studi dell'Aquila, Dipartimento di Ingegneria e Scienze dell'Informazione e Matematica, L'Aquila (Italy), E-mail address: \texttt{raffaele.folino@univaq.it}}, MAURIZIO GARRIONE\footnote{Politecnico di Milano, Dipartimento di Matematica, Milano (Italy). E-mail address: 
\texttt{maurizio.garrione@polimi.it}}, MARTA STRANI\footnote{Universit\`a Ca' Foscari, Dipartimento di Scienze Molecolari e Nanosistemi, Venezia Mestre (Italy). E-mail address: \texttt{marta.strani@unive.it}}\\

\end{center}
\vskip1cm

\begin{abstract}

In this paper we study the long time dynamics of the solutions to an initial-boundary value problem for 
a scalar conservation law with a saturating  nonlinear diffusion.
After discussing the existence of a unique stationary solution and its asymptotic stability, we focus our attention on the phenomenon of {\it metastability}, whereby the time-dependent solution develops
into a layered function in a relatively short time, and subsequently approaches
a steady state in a very long time interval. Numerical simulations
illustrate the results.

\end{abstract}

\begin{quote}\footnotesize\baselineskip 14pt 
{\bf Key words.} 
Mean curvature operator, steady states, stability, metastability.
 \vskip.15cm
\end{quote}

\begin{quote}\footnotesize\baselineskip 14pt 
{\bf AMS subject classification.} 
35K20, 35B36, 35B40, 35P15
 \vskip.15cm
\end{quote}

\pagestyle{myheadings}
\thispagestyle{plain}
\markboth{R. FOLINO, M. GARRIONE, M. STRANI}{SLOW MOTION FOR BURGERS EQUATION}

\section{Introduction}
The asymptotic behavior of solutions to evolution PDEs of the form
\begin{equation}\label{sto}
\d_t u = \mathcal P^\e[u],
\end{equation}
where $\mathcal P^\e$ is a nonlinear differential operator that depends singularly on the parameter $\e$, has been widely studied in literature. 
It is quite common that the solution to \eqref{sto} approaches a stable steady state; 
according to the time needed for this behavior to occur, we are in presence of \emph{stability} - if the convergence is exponentially fast -  or \emph{metastability} - 
if the convergence takes place in an exponentially long time interval (that becomes longer the more the parameter $\e$ approaches zero). 

\smallbreak
\noindent In this paper, we are interested in studying stability and metastability properties of the solutions to a scalar conservation law with a nonlinear diffusion; 
precisely, given $\ell>0$ and $I=(-\ell,\ell)$, we consider the following initial-boundary value problem
\begin{equation}\label{NonLBurg}
 \left\{\begin{aligned}
		\partial_t u	& =\varepsilon\,\partial_x\left(\frac{\partial_x u}{\sqrt{1+(\d_x u)^2}}\right) - \partial_x f(u),
		&\qquad &x\in I, \,t > 0,\\
 		u(\pm\ell,t)&=u_\pm, &\qquad &t >0,\\
		u(x,0)		& =u_0(x), &\qquad &x\in I.
 	 \end{aligned}\right.
\end{equation}
As concerning the second-order operator on the right-hand side, we are thus considering a \emph{mean curvature-type diffusion}. 
The prescribed mean curvature equation has been object of interest since various decades, mainly due to its natural appearance when studying the minimal surface problem. 
In the context of reaction-diffusion models, mean curvature-type diffusions were introduced as examples of \emph{saturating diffusions} in works by Rosenau and co-authors 
(see e.g. \cite{GooKurRos, KurRos, Ros}), where they were essentially motivated by the need to restore the finiteness of the energy along sharp interfaces, thus allowing discontinuous solutions. 
Indeed, some differences with the linear diffusion case are already present at the level of traveling fronts, since discontinuous steady states here appear naturally (see, e.g., \cite{GarSan, GooKurRos, KurRos}). 

As for the convective term, we assume without loss of generality that $f \in C^2(\R)$ is such that $f(0)=0$; 
the main example we have in mind is a Burgers-type convection (namely, $f(u)=u^2/2$), even if many of the results we state can be extended to more general choices. 
It is worth underlining that usually, in the case of viscous conservation laws with linear diffusion like
\begin{equation}\label{LBurg}
 \d_t u = \e \d_x^2 u - \d_x f(u),
\end{equation}
the following conditions on $f$ are required:  
 \begin{equation}\label{introconv}
 f''(u)\geq c_0>0,\qquad f(u_+)=f(u_-)
	\quad \mbox{and} \quad f'(u_+)<0<f'(u_-).
 \end{equation}
These assumptions come from the study of the formal hyperbolic equation obtained by setting $\e=0$ in \eqref{LBurg} 
and guarantee that a jump with left value $u_-$ and right value $u_+< u_-$ satisfies the entropy condition
and has speed of propagation equal to zero, as dictated by the Rankine-Hugoniot relation; 
these are necessary conditions for the existence of a unique (possibly discontinuous) entropy solution $u^0$ to the hyperbolic  equation $\d_t u = - \d_x f(u)$, 
and it has been shown \cite{Lax54, Nes96} that the solution to \eqref{LBurg} converges to $u^0$  in $L^1_{{}_{\rm loc}}$ as $\e \to 0$.
In the present paper, \eqref{introconv} is no longer necessary since, as we will see more in details in Section 2, 
the presence of the saturating diffusion forces us to impose some smallness assumptions on the boundary values (with respect to $\e$) in order to ensure the existence of a \emph{smooth} solution. 
Hence, when passing to the limit for $\e \to 0$, we have convergence to the zero function
and we can thus allow both jumps from a value $u_-$ greater than $u_+$ and the opposite, that is, we have existence of both decreasing and increasing solutions (contrary to the linear case, see for instance \cite{MS}).

\smallbreak
The dynamics for the \emph{linear} equation \eqref{LBurg} can be described as follows (see Figure \ref{fig1}): 
starting from an initial datum connecting the boundary conditions, in a relatively short time a single interface located at some point of the interval is formed and, 
subsequently, such interface starts to drift towards its asymptotic configuration (i.e. a stable steady state for the system) with a speed rate of the order $\textrm{e}^{-c/\e}$, $c > 0$. 
This is an example of {\it metastable dynamics}, where two different time scales can be spotted in the dynamics: 
the first for the formation of the internal layers, the second for the exponentially slow convergence to the equilibrium solution.
\begin{figure}[h]
\centering
\includegraphics[width=10cm,height=6cm]{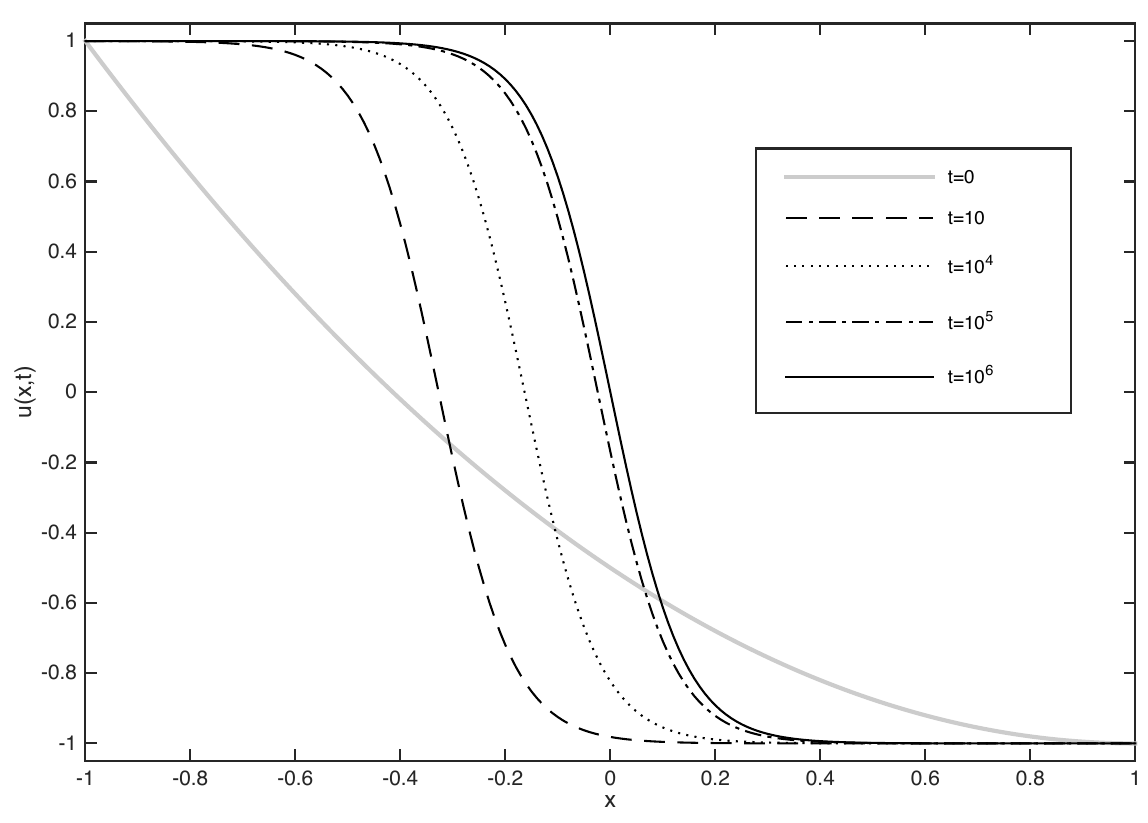}
\caption{\small{The solution to $\d_t u = \e \d_x^2 u -u\d_x u$ with  $\e=0.07$ and initial datum $u_0(x)= \frac{1}{2}x^2-x-\frac{1}{2}$ in grey. The motion of the time-dependent solution towards its asymptotic configuration, given  by the hyperbolic tangent centered in zero, takes place in an exponentially long time interval, and a metastable behavior is observed. }}\label{fig1}
\end{figure}

Many fundamental partial differential equations, coming from different fields of application, exhibit such fascinating behavior. 
Among others, we include viscous shock problems (see, for example \cite{LafoOMal94, LafoOMal95, ReynWard95, Str14bis} for viscous conservation laws, 
and \cite{BerKamSiv01, Str15, SunWard99} for Burgers type equations), phase transition problems described by the Allen-Cahn equation,
with the fundamental contributions \cite{CarrPego89, FuscHale89} and the most recent references \cite{OttoRezn06,Str13}, 
and the Cahn-Hilliard equation studied in \cite{AlikBateFusc91} and \cite{Pego89}. 
We finally quote some recent papers on metastability for hyperbolic versions of both the Allen-Cahn and the Cahn-Hilliard equation \cite{Fol, FLM1, FLM2}.

Motivated by the behavior of the solutions to the linear equation \eqref{LBurg}, we thus wonder if the dynamics for \eqref{NonLBurg} presents similar features. 
In order to examine if metastability occurs for such a problem, in the first part of this paper we will focus on the existence and uniqueness of a (monotone and classical) steady state for problem \eqref{NonLBurg}: 
it turns out that, due to the mean curvature type diffusive term, this is ensured only if some additional assumptions on the boundary conditions are imposed. 
In particular, our first result gives a description of the solution to
\begin{equation}\label{NonLBurgstaz}
 \left\{\begin{aligned}
		\varepsilon\,\partial_x\left(\frac{\partial_x u}{\sqrt{1+(\d_x u)^2}}\right) & =  \partial_x f(u),
		&\quad &x\in I,\\
 		u(\pm\ell)&=u_\pm,
 	 \end{aligned}\right.
\end{equation}
and can be sketched as follows.
\begin{theorem}
Fix $f\in C^1(\R)$ and $\varepsilon > 0$. 
Then, there exists a positive constant $C$ (whose explicit expression is given in Section \ref{sezSS}) 
such that a unique {\it decreasing} (resp., increasing) solution to \eqref{NonLBurgstaz} connecting $u_- > u_+$ (resp., $u_- < u_+$) exists if and only if 
$$
\max_{u \in [u_-, u_+]} f(u) - \min_{u \in [u_-, u_+]} f(u) < \e \qquad {\rm and } \qquad \ell>C.
$$
\end{theorem}
We notice that, when the flux function $f$ is chosen, the previous conditions may turn into smallness assumptions on the choice of the boundary values $u_\pm$
(compare, for instance, with \cite{CheKurRos05}).
\smallbreak
Once the existence of a solution to \eqref{NonLBurgstaz} is proven, we turn our attention to the investigation of its stability properties. 
We will see that the occurrence of either \emph{stability} or \emph{metastability} depends again on the choice of the flux function and of the boundary conditions. 

Precisely, in our second result (Theorem \ref{propstabilita1}) we show that, if the boundary data $u_\pm$ are sufficiently small, then stability of the unique (increasing or decreasing) steady state occurs, that is, 
\begin{equation*}
\Vert u(\cdot, t)-\bar u \Vert_{{}_{L^2}} \leq e^{-C_\e t}\Vert u_0-\bar u\Vert_{{}_{L^2}},
\end{equation*}
where $u$ and $\bar u$ are  the solution to \eqref{NonLBurg} and \eqref{NonLBurgstaz} respectively, and $C_\e$ is a positive constant depending on $\e$. 
Here the $L^2$--distance between the solution to \eqref{NonLBurg} and the steady state thus decays to zero {\it exponentially} as $t \to +\infty$. 

We then analyze the long-time dynamics of the solution to \eqref{NonLBurg} when the assumptions of Theorem \ref{propstabilita1} are not fulfilled; 
in order to prove that the aforementioned metastable behavior appears, the strategy we use closely resembles the one first performed in \cite{MS}, 
and is based on the construction of a {\it one-parameter  family of approximate stationary solutions}, denoted here by $\{ U^\e(x;\xi) \}_{\xi \in I}$ (for more details, see hypothesis {\bf H1} in Section~\ref{meta}). 
The key point of the strategy is to linearize the original system around the generic element $U^\e(x;\xi)$, 
where the parameter $\xi$ describes the reduced dynamics along this family and can be thought as the location of the internal interface of the solution. 
By letting $\xi= \xi(t)$ depend on time, we write the solution $u$ as
\begin{equation}\label{lin}
u(x,t)= U^\e(x;\xi(t))+v(x,t),
\end{equation}
and we describe the convergence of the solution $u$ towards its asymptotic configuration by following the evolution of $\xi(t)$ towards its equilibrium location 
(named here $\bar \xi$, so that $U^\e(x;\bar \xi)$ is an exact steady state for the problem).
Going deeper in details, in Section \ref{meta} we show  that the perturbation $v$ is small and the dynamics of $\xi$ can be described by the ODE $\xi'=\theta^\e(\xi)$, 
where $\theta^\e$ is a monotone decreasing function, satisfying $\theta^\e(\bar\xi)=0$ (for the precise definition of $\theta^\e$, see Section \ref{meta}).
Because of the decomposition \eqref{lin}, we have that the solution $u$ to \eqref{NonLBurg} is drifting to its equilibrium configuration with a speed dictated by the speed rate of convergence of $\xi$ towards $\bar \xi$; 
in particular, we show that in the case of a Burgers-type flux function $f$, if the boundary data $u_\pm$ are properly chosen, then $\theta^\e\to 0$ exponentially as $\e \to 0$, 
so that such convergence is much slower as $\e$ becomes smaller.
\medbreak
We close this introduction with a short plan of the paper. 
Section 2 is devoted to the study of the stationary problem \eqref{NonLBurgstaz}; 
precisely, in Theorems \ref{teoincr} and \ref{teodecr} we prove existence and uniqueness of an increasing and of a decreasing steady state. 
In Section 3, we deal with the well-posedness of problem \eqref{NonLBurg} and we give a result of asymptotic stability for the steady states (the aforementioned Theorem \ref{propstabilita1}). 
Finally, Sections \ref{meta:num} and \ref{meta} concern the study of the metastable behavior of the solution to \eqref{NonLBurg}, 
which is shown through some numerical evidences, as well.  

\section{Existence and uniqueness of the steady state}\label{sezSS}

\subsection{The problem on the whole line} \label{onde}
Before turning our attention to the steady states for problem \eqref{NonLBurg}, let us briefly comment about traveling wave-type solutions ($I=\mathbb{R}$) as a possible further motivation for our study.

In particular, in \cite{GarStr} the attention was devoted to \emph{heteroclinic wave fronts}, namely solutions $u(x, t) = v(x+ct)$ of a general reaction-convection-diffusion equation 
$$
\begin{aligned}
		\partial_t u	& =\partial_x\left(\frac{\partial_x u}{\sqrt{1+(\d_x u)^2}}\right) - \partial_x f(u) + g(u),
		&\qquad &x\in \R, \, t > 0, \\
 	 \end{aligned}
$$
defined on the whole $\mathbb{R}$ and connecting two different equilibria at $\pm \infty$ (namely, $v(-\infty)=u^-$, $v(+\infty)=u^+$). 
The numbers $c \in \mathbb{R}$ for which a solution of this type exists are called \emph{admissible speeds}. 
Notice that these solutions may somehow be seen as limits of Dirichlet and Neumann solutions on $[-L, L]$, for $L \to +\infty$.

When the reaction term is not present, it was shown in \cite[Section 4.2]{GarStr} that the unique admissible speed for which there exists a decreasing traveling front connecting $u_-$ and $u_+$ is given by 
$
c = \frac{f(u_-)-f(u_+)}{u_+ - u_-}, 
$
provided that $u_-$ and $u_+$ are sufficiently close (see also \cite{CheKurRos05}).  
Now, in presence of a parameter $\e$ in front of the diffusion, the problem is brought back to the study of the ODE 
$$
\varepsilon \left(\frac{v'}{\sqrt{1+(v')^2}}\right)' - (c+ f'(v)) v' = 0;
$$
since it immediately follows that 
$
\varepsilon \frac{v'}{\sqrt{1+(v')^2}} - cv -f(v)
$
is constant, it necessarily has to be again $c = \frac{f(u_-)-f(u_+)}{u_+ - u_-}$ and it turns out that 
$$
\frac{v'}{\sqrt{1+(v')^2}} = \frac{\frac{f(u_-)-f(u_+)}{u_+ - u_-} v + f(v)}{\varepsilon}.
$$
Being the left-hand side bounded and $\varepsilon$ small, 
this can be the case only if $u_-$ and $u_+$ are chosen sufficiently small (even closer with respect to the case $\e=1$), so that $f(v)$ will be small as well. 
This is a consequence of the fact that the saturation produces a weakness in the diffusion part, 
which is not able any more to counterbalance the convection with a regular solution if the states to be connected are too large. 
In the next subsection, we will notice this behavior also on a bounded interval (intuitively, we may somehow think that some information therein may be recovered by suitably truncating and rescaling a wave front).

\subsection{Steady states on a bounded interval}\label{stst}
We now consider the stationary problem for \eqref{NonLBurg}; 
we will deal with both increasing and decreasing stationary solutions, under general assumptions depending only on the two values 
$$
m=\min_{u \in [u_-, u_+]} f(u) \qquad {\rm and} \qquad M=\max_{u \in [u_-, u_+]} f(u)
$$ 
(notice that it will be $u_- > u_+$ or $u_- < u_+$ according to the monotonicity of the steady state).  
{As we will see in the following, such assumptions will be read as  restrictions on the choice of the boundary data}. 

Stationary solutions to \eqref{NonLBurg} solve the equation 
\begin{equation}\label{staz}
\frac{\e \d_x u}{\sqrt{1+(\d_x u)^2}}= f(u) + C,
\end{equation}
where $C \in \R$ is an integration constant that is uniquely determined once the boundary conditions $u(\pm \ell)=u_\pm$ are imposed. 
For the sake of briefness, we limit ourselves to consider the case  when either $\d_x u > 0$ or $\d_x u < 0$ everywhere. 

In case we look for \emph{increasing} steady states, it has to be $u_- < u_+$ and we have the following statement. 
\begin{theorem}\label{teoincr}
Fix $f\in C^1([u_-,u_+])$ and $\e > 0$.
Then, there exists a unique increasing (smooth) stationary solution of \eqref{NonLBurg} if and only if 
\begin{equation}\label{ipodatibordo}
	M-m < \e \qquad \textrm{ and } \qquad 2\ell>c_I:=\int_{u_-}^{u_+} \frac{\sqrt{(M-f(u))(f(u)+2\e-M)}}{f(u)+\e-M} \, du.
\end{equation}
In particular, one has
\begin{equation*}
	0<c_I\leq\frac{\sqrt2\e}{\e-(M-m)}(u_+-u_-).
\end{equation*}
\end{theorem}
The two conditions in \eqref{ipodatibordo} are consistent with the discussion in the above subsection, 
enlightening once more the weakness of mean curvature-type diffusions (notice that this happens, with a milder control, also for $\e=1$). 
In particular, once $f$ is given, \eqref{ipodatibordo} forces us to choose ``sufficiently small''  boundary data.

\begin{proof}[Proof of Theorem \ref{teoincr}]
Since by assumption $\partial_x u > 0$ for every $x \in I$, it follows from \eqref{staz} that $C + m > 0$; 
on the other hand, since the left-hand side in \eqref{staz} is always smaller than $\e$, it has to be $C < \e - M$. 
It follows that the first condition in \eqref{ipodatibordo} has to be satisfied. 
We then search for $-m < C < \e-M$ such that \eqref{staz} has an increasing solution $u$; solving for $\partial_x u$ therein gives 
\begin{equation*}
	\d_x u = \frac{f(u)+C}{\sqrt{\e^2-(f(u)+C)^2}}. 
\end{equation*}
Stationary solutions are then implicitly defined by
\begin{equation}\label{ISS}
\int_{u_-}^{u(x)} \frac{\sqrt{\e^2-(f(u)+C)^2}}{f(u)+C} \, du= \int_{-\ell}^x \, ds = x+l.
\end{equation}
Setting 
\begin{equation}\label{definphi}
\Phi(C)=\int_{u_-}^{u_+} \frac{\sqrt{\e^2-(f(u)+C)^2}}{f(u)+C} \, du,
\end{equation} 
we observe that $\Phi$ is well defined, strictly positive and decreasing as a function of $C$. 
By the Monotone Convergence Theorem and the regularity of $f$, it holds 
\begin{equation}\label{eq:limits}
	\begin{aligned}
	\lim_{C \to -m} \Phi(C)&=\int_{u_-}^{u_+} \frac{\sqrt{\e^2-(f(u)-m)^2}}{f(u)-m} \, du=+\infty , \\ 
	\lim_{C \to \e-M} \Phi(C)&=\int_{u_-}^{u_+} \frac{\sqrt{(M-f(u))(f(u)+2\e-M)}}{f(u)+\e-M} \, du=:c_I. 
\end{aligned}
\end{equation}
Thus, there exists (a unique) $C$ with $-m < C < \e - M$ such that $\Phi(C)=2\ell$ (and \eqref{ISS} is satisfied) if and only if \eqref{ipodatibordo} is fulfilled.
\end{proof}

Observe that the first limit of \eqref{eq:limits} is equal to $+\infty$ because of the regularity of $f$;
in case $f\notin C^1$ it could be finite and we would have also a restriction from above for the choice of $\ell$ in \eqref{ipodatibordo}.

On the other hand, reasoning on \emph{decreasing} steady states, we have to assume that $u_- > u_+$ and the following result holds. 
\begin{theorem}\label{teodecr}
Fix $f\in C^1([u_+,u_-])$ and $\e > 0$.
Then, there exists a unique decreasing (smooth) stationary solution of \eqref{NonLBurg} if and only if 
$$
M-m < \e \qquad \textrm{ and } \qquad 2\ell>c_D:=\int_{u_+}^{u_-} \frac{\sqrt{(f(u)-m)(2\e+m-f(u))}}{m+\e-f(u)} \,du.
$$
In particular, one has
\begin{equation*}
	0<c_D\leq\frac{\sqrt2\e}{\e-(M-m)}(u_+-u_-).
\end{equation*}
\end{theorem}

The proof of Theorem \ref{teodecr} is similar to the one of Theorem \ref{teoincr}.
Notice that the condition $M-m < \e$ has to hold true in both the increasing and the decreasing case, meaning that regular transitions are possible only if the boundary data are sufficiently small.
\begin{example}{\rm Assuming $f(u)=u^2/2$ (namely, $f$ is a Burgers flux) and $u_\pm=\pm u_*$ or $u_\pm=\mp u_*$ for some $u_* >0$, 
according to whether we search for increasing or decreasing steady states, we have $m=0$,  and the assumptions of Theorems \ref{teoincr} and \ref{teodecr} read as
$$
\frac{u^2_*}{2} < \e \quad \textrm{ and } \quad   \begin{cases} 2 \ell > c_I \quad \mbox{for the increasing case},  \\
 2\ell > c_D\quad \mbox{for the decreasing case},
\end{cases} 
$$
where $c_I$ and $c_D$ go to zero as $\e \to 0$.
}
\end{example}

{Finally, we notice that in fact non-monotone solutions of \eqref{NonLBurgstaz} cannot exist. Rewriting the differential equation in \eqref{NonLBurgstaz} as
\begin{equation}\label{sist}
 \left\{\begin{aligned}
		u' & = \frac{v}{\sqrt{1-v^2}}  \\
		v'&=\frac{f'(u)}{\varepsilon} \frac{v}{\sqrt{1-v^2}},
 	 \end{aligned}\right.
\end{equation}
we see indeed that for this system of ODEs each point of the axis $\{v=0\}$ is an equilibrium (i.e., it is invariant for the dynamics). If a non-monotone solution $u$ of \eqref{NonLBurgstaz} existed, in the phase plane $(u,v)$ it would correspond to a solution of \eqref{sist} lying in both the regions $\{v >0\}$ and $\{v<0 \}$, but thus it would cross the axis $v=0$ and this is not possible because of the uniqueness of the solution.}


\section{Stability properties of the steady states} 
We are here interested in the asymptotic behavior of the classical solutions to \eqref{NonLBurg} for large times: 
the main goal is to prove their convergence to the steady state found in Section \ref{stst} for $t \to +\infty$.
\subsection{Global existence for the initial-boundary value problem.} Of course, the first issue that has to be addressed is to study the solvability of the Cauchy-Dirichlet problem 
\begin{equation}\label{burgers}
 \left\{\begin{aligned}
		\partial_t u	& =\varepsilon\,\partial_x\left(\frac{\partial_x u}{\sqrt{1+(\d_x u)^2}}\right) - \partial_x f(u),
		&\qquad &x\in I, t > 0,\\
 		u(\pm\ell,t)&=u_\pm, &\qquad &t >0,\\
		u(x,0)	& =u_0(x), &\qquad &x\in I.
 	 \end{aligned}\right.
\end{equation}
Finding a global existence result is here not a trivial matter. 
Indeed, quasilinear differential equations ruled by the curvature operator (or more in general by saturating diffusions) may display blow-up of the solutions, see \cite{GooKurRos,KurRos}. In this direction, it is worth mentioning the recent paper \cite{Zhang}, where the author shows  that the Dirichlet problem for a curvature flow with driving force blows up with large initial data.
\smallbreak
When dropping the boundary requirement and considering the Cauchy problem on the whole real line with ``small'' initial data, 
some results were obtained in \cite[Theorem 3.3]{KurRos}, where the authors remark the difficulties in giving a well-posedness result for general initial data. 
We here recall such a statement.
\begin{theorem}\label{EUG}
Consider the problem 
\begin{equation}\label{problema}
\left\{
\begin{array}{l}
\displaystyle \partial_t u = \e \,\partial_x \left(\frac{\partial_x u}{\sqrt{1+\partial_x u^2}}\right) + \partial_x f(u) \vspace{0.1cm}\\
u(x, 0) = u_0(x).
\end{array}
\right.
\end{equation}
If $u_0 \in C^3$ and there exists $\alpha > 0$ such that 
\begin{equation*}
\e \left\Vert \frac{\partial_x u_0}{\sqrt{1+(\partial_x u_0)^2}} \right\Vert_{{}_{L^\infty}} + 2 \Vert f(u_0) \Vert_{{}_{L^\infty}} \leq \alpha < \e,
\end{equation*}
then there exists a unique global classical solution $u(x, t) \in C^{2, 1}$ of \eqref{problema}. 
\end{theorem}
Theorem \ref{EUG} was proved using the vanishing viscosity method, i.e., 
inserting a small regularizing part $\delta \partial_x^2 u$ inside the equation and providing uniform estimates on the corresponding sequence of (unique, regular and global) solutions $u_\delta$ for small $\delta$'s. 
Passing to the limit for $\delta\to0^+$, the authors obtained the unique classical solution of \eqref{problema}.
In particular, in \cite{KurRos} it is shown that the solution of \eqref{problema} preserves the smallness condition satisfied by the initial datum and then the derivative $\partial_x u$ remains bounded for all times.
A similar result holds in the case of classical solutions to the IBVP \eqref{burgers}.
\begin{proposition}\label{prop:aprioriux}
Let $\e, T > 0$ be fixed and let $u(\cdot,t)\in C^3(I)$ be a classical solution of the IBVP \eqref{burgers}, for $t\in[0,T]$.
If
\begin{equation}\label{eq:smallness}
\e\left\|\frac{u'_0}{\sqrt{1+(u'_0)^2}}\right\|_{{}_{L^\infty}}+2\Vert f(u_0)\Vert_{{}_{L^\infty}}\leq\alpha<\e,
\end{equation}
then 
\begin{equation}\label{stimaux}
\Vert \partial_x u(\cdot,t)\Vert_{{}_{L^\infty}}\leq h^{-1}(\alpha/\e),
\end{equation}
for any $t\in[0,T]$, being $h(s)=\frac{s}{\sqrt{1+s^2}}$.
\end{proposition}
Let us compare conditions \eqref{eq:smallness} and \eqref{ipodatibordo}: since by \eqref{eq:smallness} one has
\begin{equation*}
M-m\leq 2\|f(u_0)\|_{{}_{L^\infty}}<\e,	
\end{equation*}
assumption \eqref{eq:smallness} implies condition \eqref{ipodatibordo}, which we had to impose in order to obtain the existence of regular steady states. 
\begin{proof}[Proof of Proposition \ref{prop:aprioriux}]
Similarly as in \cite{KurLevRos, KurRos}, we define the function
\begin{equation}\label{eq:z}
z(x,t):=\frac{\varepsilon\partial_x u(x,t)}{\sqrt{1+(\d_x u(x,t))^2}}- f(u(x,t)),
\end{equation}
where $u$ is the classical solution of \eqref{burgers}.
Therefore, we can rewrite the first equation of \eqref{burgers} as
\begin{equation*}
\d_t u=\d_x z.
\end{equation*}
Differentiating equation \eqref{eq:z} with respect to $t$, we deduce
\begin{equation}\label{eq:z_t}
\d_t z=\frac{\e \d^2_xz}{\left\{1+(\d_x u(x,t))^2\right\}^{3/2}}-f'(u)\d_x z,
\end{equation}
with homogenous Neumann boundary conditions $\partial_x z(\pm \ell, t)=0$.
Equation \eqref{eq:z_t} and the differential equation in \eqref{burgers} remain parabolic for $t\in[0,T]$  so that,
thanks to the maximum principle, it holds 
\begin{equation}\label{aprioriu} 
\Vert u(\cdot,t)\Vert_{{}_{L^\infty}}\leq\Vert u_0\Vert_{{}_{L^\infty}},
\end{equation}
for any $t\in[0,T]$; moreover, since  the constants $C_\pm:= \pm \Vert z(\cdot,0)\Vert_{{}_{L^\infty}}$ are solutions to \eqref{eq:z_t}, by a standard comparison principle it follows
\begin{equation*}
\Vert z(\cdot,t)\Vert_{{}_{L^\infty}}\leq\Vert z(\cdot,0)\Vert_{{}_{L^\infty}}=\left\Vert \frac{\varepsilon u'_0}{\sqrt{1+(u'_0)^2}}- f(u_0)\right\Vert_{{}_{L^\infty}},
\end{equation*}
for any $t\in[0,T]$. 
Therefore, using \eqref{eq:smallness} and \eqref{aprioriu}, we obtain
\begin{equation*}
\varepsilon\left\Vert\frac{\partial_x u(x,t)}{\sqrt{1+(\d_x u(x,t))^2}}\right\Vert\leq\alpha<\e,
\end{equation*}
for any $t\in[0,T]$. 
Since the function $h(s)=\frac{s}{\sqrt{(1+s^2)}}$ is increasing and satisfies $h(\pm\infty)=\pm1$, we conclude
\begin{equation*}
\Vert \partial_x u(\cdot,t)\Vert_{{}_{L^\infty}}\leq h^{-1}(\alpha/\e),
\end{equation*}
for any $t\in[0,T]$, and the proof is complete.
\end{proof}
\begin{remark}{\rm Notice that the constant $h^{-1}(\alpha/\e)$ in \eqref{stimaux} can be chosen independently on $\e,T$ if the initial datum $u_0$ is sufficiently small.
Indeed, if we choose $\alpha=\frac34\e$ in \eqref{eq:smallness}, then the estimate \eqref{stimaux} holds with $h^{-1}(3/4)$.
}\label{rem:C0}
\end{remark}

Proposition \ref{prop:aprioriux} provides an a priori estimate on the spatial derivative of the classical solution to \eqref{burgers}. 
This result can be used in order to obtain a global existence result;
indeed, the local existence of a classical solution (as coming, e.g., from \cite[Theorem 8.2]{Lie}), 
together with the bound \eqref{stimaux}, allows us to prove the existence of a global classical solution satisfying the estimates \eqref{stimaux} and \eqref{aprioriu} for any $t\geq0$. 
On the contrary, the behavior of the solutions when condition \eqref{eq:smallness} is not satisfied has been investigated in \cite{GooKurRos}, 
both for the Cauchy problem \eqref{problema} and the IBVP \eqref{burgers}. 
In this case, the solution may develop discontinuities in a finite time; 
in particular, for certain flux functions $f$ and large initial data $u_0$, there exists a finite breaktime $T > 0$ such that 
\begin{equation*}
	\lim_{t\to T^-}\Vert \d_x u(\cdot,t)\Vert_{{}_{L^\infty}}=+\infty.
\end{equation*}
Moreover, it was numerically shown in \cite{GooKurRos} that both continuous and discontinuous steady states are strong attractors of a wide class of initial data.

Anyway, a complete discussion about the well-posedness of the IBVP \eqref{burgers} and about the stability of discontinuous steady states is beyond the scope of this paper, 
since here we are interested in studying the long time behavior of classical solutions and their metastable dynamics.
Therefore, from now on we assume that the initial datum $u_0$ satisfies \eqref{eq:smallness} and that it is sufficiently smooth, so that problem \eqref{burgers} has a unique global classical solution. 
Moreover, we assume for simplicity that $u_0$ is a strictly monotone function. 
These conditions, as well as \eqref{eq:smallness} with respect to well-posedness, are not necessary; 
in the next section, we will show numerical simulations where the initial datum is either non-monotone or discontinuous, 
but the solution becomes monotone (cf. Figure \ref{fig3new}) and continuous (see right picture in Figure \ref{quellabella}) in finite time.

The following proposition, which will be useful soon after, shows that in our setting the monotonicity is preserved in time.
\begin{proposition}\label{crescente}
Let $u(x, t)$ be a classical solution of \eqref{burgers}, with $u_0 \in C^3(I)$ monotone increasing (decreasing) and satisfying \eqref{eq:smallness}. 
Then, for every $t > 0$, $u(\cdot, t)$ is monotone increasing (decreasing). 
 \end{proposition}
\begin{proof}
We prove the statement for $u_0$ increasing.
Observe that from \eqref{aprioriu} it follows that
\begin{equation}\label{aprioriu-monotone}
	u_-\leq u(x,t)\leq u_+, \qquad \qquad \forall\, x\in [-\ell,\ell], \, t\geq0.
\end{equation}
By differentiating with respect to $x$ the differential equation in \eqref{problema}, we obtain that $w = \d_x u$ solves
\begin{equation}\label{eqw}
\d_t w =\e \frac{ \d_x^2 w}{(1+w^2)^{3/2}} - 3\e \frac{  w (\d_x w)^2}{(1+w^2)^{5/2}}-f'(u) \d_x w-f''(u)w^2.
\end{equation}
Equation \eqref{eqw} is parabolic and both $w=0$ and $w=\d_x u$ are solutions; 
moreover, $w(x, 0) \geq 0$ by assumption and $w(\pm\ell, t)=u_x(\pm \ell, t) \geq 0$ for every $t$, otherwise \eqref{aprioriu-monotone} would be violated. 
By the comparison principle \cite[Theorem 9.7]{Lie}, then, $w(x, t) \geq 0$ for every $x \in [-\ell, \ell]$ and $t > 0$,
namely $u(\cdot, t)$ is increasing for all $t>0$. In case $u_0$ is decreasing, a similar argument (notice that \eqref{aprioriu-monotone} here holds with reverse signs) yields the conclusion. 
\end{proof} 


\subsection{Stability of the increasing steady state} 

We now deal with the stability properties of the increasing steady state $u_I$ implicitly defined by the relation \eqref{ISS}; 
precisely, we prove that, in presence of sufficiently small boundary data, the $L^2$-distance between the classical solution of \eqref{burgers} and $u_I$ goes to zero as $t \to \infty$.

\begin{theorem}\label{propstabilita1}
Fix $\e > 0$ and 
denote by $u$ the classical solution to \eqref{burgers}, where the initial datum $u_0 \in C^3(I)$ is strictly monotone and satisfies \eqref{eq:smallness}. 
Assume moreover that \eqref{ipodatibordo} holds. Then, there exists a positive constant $\bar c $ (that depends on $\e$ and can be explicitly computed) 
such that if
\begin{equation}\label{ipodatibordo2}
\sup_{u \in [u_-, u_+]} |f'(u)| \leq \bar c \, \e,
\end{equation}
then there exists $K_\e >0$ such that
\begin{equation*}
\Vert u(\cdot, t)- u_I\Vert_{{}_{L^2}} \leq e^{-K_\e t} \Vert u_0- u_I\Vert_{{}_{L^2}}.
\end{equation*}
\end{theorem}
The proof of Theorem \ref{propstabilita1} essentially relies on Proposition \ref{crescente} and on a priori estimates using the definition of weak solution.
\begin{proof}[Proof of Theorem  \ref{propstabilita1}] 
We first notice that, thanks to Proposition \ref{crescente}, the solution $u$ of \eqref{burgers} is increasing in $x$, since $u_0$ is increasing.
We set $w=u(x, t) - u_I(x)$, where $u_I$ is the strictly increasing stationary solution implicitly defined in \eqref{ISS}, which satisfies 
\begin{equation}\label{derss}
\d_x u_I = \frac{f(u)+C_I}{\sqrt{\e^2-(f(u)+C_I)^2}}, 
\end{equation}
for a suitable $C_I$ such that $0 < C_I < \e - f(u_\pm)$. 
By integrating the equation against the test function $\varphi=w$ we obtain
\begin{equation}\label{1stima}
\begin{aligned}
\frac{1}{2}\frac{d}{dt}&\Vert w \Vert^2_{{}_{L^2}}(t)-\int_{-\ell}^\ell\left(f(u)-f( u_I) \right) \d_x w \, dx  \\
 &\small +{\e}\int_{-\ell}^\ell  \frac{(\d_x u+ \d_x  u_I)(\d_x w)^2}{\sqrt{1+(\d_x u_I)^2}\sqrt{1+(\d_xu)^2}\left(\d_x u \sqrt{1+(\d_x u_I)^2}+ \d_x u_I \sqrt{1+(\d_xu)^2}\right)} \, \, dx=0.
\end{aligned}
\end{equation}
As concerning the second term in \eqref{1stima}, by using the H\"older inequality we have 
\begin{equation*}
\begin{aligned}
\left\vert \int_{-\ell}^{\ell} \Big(f(u)-f(u_I)\Big) \partial_x w \, dx \right \vert &\leq \left(  \sup_{u \in [u_-, u_+]} |f'(u)| \right) \int_{-\ell}^\ell \vert w\partial_x w \vert \, dx  \\
&\leq c_p\left(   \sup_{u \in [u_-, u_+]} |f'(u)| \right) \Vert \partial_x w \Vert^2_{{}_{L^2}},
\end{aligned}
\end{equation*}
where $c_p=(2\ell/\pi)^2$ is the constant appearing in the Poincar\'e inequality $ \Vert w \Vert^2_{{}_{L^2}} \leq c^2_p \Vert \d_x w \Vert^2_{{}_{L^2}}$.
As for the other integral term, set
\begin{equation*}
A_\e := \max \Bigg\{ \sqrt{1+C_0^2} \, , \sqrt{1+\Vert \d_x u_I\Vert^2_{{}_{L^\infty}}} \, \Bigg\}, 
\end{equation*}
where $C_0=h^{-1}(\alpha/\e)$ (see \eqref{stimaux}). 
In view of Proposition \ref{crescente}, we have
\begin{equation*}
{\sqrt{1+(\d_x u_I)^2}\sqrt{1+(\d_xu)^2}\left(\d_x u \sqrt{1+(\d_x u_I)^2}+ \d_x u_I \sqrt{1+(\d_xu)^2}\right)} \leq A^2_\e \left( A_\e [\d_x u + \d_x u_I]\right),
\end{equation*}
so that for the third term in \eqref{1stima} we infer
$$
{\e}\int_{-\ell}^\ell  \frac{(\d_x u+ \d_x  u_I)(\d_x w)^2}{\sqrt{1+(\d_x u_I)^2}\sqrt{1+(\d_xu)^2}\left(\d_x u \sqrt{1+(\d_x u_I)^2}+ \d_x u_I \sqrt{1+(\d_xu)^2}\right)}\geq \frac{\e}{B_\e} \Vert \partial_x w \Vert^2_{{}_{L^2}},
$$
being $B_\e := A_\e^3 $.
Hence, \eqref{1stima} becomes
$$
\frac{1}{2}\frac{d}{dt} \Vert w\Vert^2_{{}_{L^2}}(t) + \Big(\frac{\e }{B_\e} -\Big(\frac{2\ell}{\pi}\Big)^2 \sup_{u \in [u_-, u_+]} |f'(u)| \Big)  \Vert \d_x w\Vert_{{}_{L^2}}^2(t) \leq 0.
$$
Choosing
$$
\bar{c} < \Big(\frac{\pi}{2\ell}\Big)^2 \frac{1}{B_\e}, 
$$
in view of assumption \eqref{ipodatibordo2} we have
$$
\frac{\e }{B_\e}-\Big(\frac{2\ell}{\pi}\Big)^2\sup_{u \in [u_-, u_+]} |f'(u)|  > 0
$$
so that we can use again the Poincar\'e inequality ending up with  
$$
\frac{1}{2}\frac{d}{dt} \Vert w\Vert^2_{{}_{L^2}}(t) + K_\e \Vert  w\Vert_{{}_{L^2}}^2(t) \leq 0, 
$$
where
$$K_\e:={c_p^{-2}}\left(\frac{\e }{B_\e}-\Big(\frac{2\ell}{\pi}\Big)^2\sup_{u \in [u_-, u_+]} |f'(u)|\right).$$ 
The statement follows from the standard comparison principle for ODEs. 
\end{proof}

A drawback of the above proof is that $\bar{c}$ does not have a direct and simple estimate. 
Indeed, one should give an explicit estimate of $\partial_x u_I$ (using for instance \eqref{derss}), and to this end the constant $C_I$ for which $\Phi(C_I)=2\ell$ should be controlled. 
This appears in fact quite involved and can in general be done only numerically. 
However, we can give a rough estimate of $\bar{c}$ in some cases: for instance, if $f$ is positive out of $0$, from the equality $\Phi(C_I)=2\ell$ we deduce that 
$$
C_I \leq \frac{\e (u_+ - u_-)}{2\ell}. 
$$
Hence, if for instance 
$$
f(u) < \e \Big(\frac{1}{2} - \frac{u_+ + u_-}{2\ell}\Big),
$$
we can deduce from \eqref{derss} that 
$$
\partial_x u_I < \frac{1}{\sqrt{3}}.
$$
Setting 
$M_0=1+C_0^2$ we thus infer 
that $A_\e = \max\{\sqrt{M_0} \, , 2\sqrt{3}/3\}$, so that it is sufficient to choose 
$$
\bar{c} < \Big(\frac{\pi}{2\ell}\Big)^2 \frac{1}{\max\{M_0^{3/2}, 8\sqrt{3}/9\}},
$$
and the constant on the right hand side can be chosen independently on $\e$ (see Remark \ref{rem:C0}).
This choice would also be reflected in a lower bound for $K_\e$.
The more $\ell$ approaches $c_I^+/2$, the more $\partial_x u_I$ will be large, since there will be less room to connect $u_-$ and $u_+$. 
\\
Finally, notice that the argument may be repeated similarly when considering decreasing solutions. 

\smallbreak
\noindent
{\bf Comments on the assumptions on the flux function $f$.} We conclude this section with some comments on the assumptions \eqref{ipodatibordo} and \eqref{ipodatibordo2}. 

\smallbreak
\noindent
As already remarked, such assumptions have to be read as {\it smallness} hypotheses on the boundary data; 
indeed, they are clearly satisfied for any $f \in C^2(\R)$ such that $f(0)=f'(0)=0$,  if $u_\pm$ are sufficiently small.
For example, if $f$ is a   power law of the form $f(u)= \k u^{p}$, $p > 1$, then \eqref{ipodatibordo} and \eqref{ipodatibordo2} are respectively satisfied  if
\begin{equation*}
\begin{aligned}
 \max\left\{|u_+|, |u_-|\right\} < \left( \frac{\e}{2\k} \right)^{1/p}\quad {\rm and} \quad  \max\left\{|u_+|, |u_-|\right\} < \left(\frac{ \bar c \, \e}{ p \,  \k }\right)^{1/(p-1)}.
\end{aligned}
\end{equation*}
We notice that, since $\e$ is small, condition \eqref{ipodatibordo2} is stronger than \eqref{ipodatibordo}; 
for example, in the case of a Burgers flux $f(u)=u^2/2$, we need to ask $|u_\pm|  <  \bar c \, \e$, which also implies $|u_\pm| <  \sqrt{\e}$ (again by the smallness of $\e$). 
Hence, when $  \bar c \, \e<|u_\pm|< \sqrt{\e}$, we know that a unique steady state exists, but Theorem \ref{propstabilita1} does not prove its stability.
We will focus the attention on this issue in the following section, where we will show that in some cases the steady state is indeed \emph{metastable}.
\smallbreak
\noindent We also observe that in the case of a linear flux $f(u)= \k u$, the second integral in \eqref{1stima} is zero; hence, we no longer need to require assumption \eqref{ipodatibordo2}.
\smallbreak
\noindent
Finally, we observe that all the results of this section hold also for $\e$ large; in this case condition \eqref{ipodatibordo} is milder and may imply \eqref{ipodatibordo2}.

\section{The metastable dynamics: numerical evidences} \label{meta:num}
In this section, we illustrate some numerical simulations for the time-dependent solution  to the following initial-boundary value problem
\begin{equation}\label{burgers2}
 	\left\{\begin{aligned}
		\partial_t u	& =\varepsilon\,\partial_x\left(\frac{\partial_x u}{\sqrt{1+(\d_x u)^2}}\right) - \partial_x f(u),
		&\qquad &x\in I, t > 0,\\
 		u(\pm\ell,t)&=u_\pm, &\qquad &t >0,\\
		u(x,0)		& =u_0(x), &\qquad &x\in I,
 	 \end{aligned}\right.
\end{equation}
aiming at showing that a metastable behavior appears if appropriately choosing the data. 

The flux function $f$ is here and throughout the rest of the paper assumed to satisfy the following additional hypotheses:
\begin{equation}\label{ipof}
	f''(u)\geq c_0>0 \quad \mbox{for every} \ \ u,
	\qquad f(u_+)=f(u_-),
\end{equation}
being the principal example we have in mind the case of a Burgers flux $f(u)=u^2/2$. 
Notice that this convexity assumption will be needed in order to observe a metastable behavior, while it is not necessary for the existence. 
We further notice that \eqref{ipof} and \eqref{ipodatibordo} give $f(u_\pm) < \e$. 
Here and throughout this section, we will actually choose $f(u_\pm)=\e/2$; 
indeed, according to Theorem \ref{propstabilita1}, if the boundary data are taken too small, 
then the steady state is stable but not metastable, as already remarked (see also Remark \ref{remarkstm}).

To start with, numerical simulations suggest that the occurrence of a metastable behavior strongly depends on the initial conditions (the same phenomenon has been observed in \cite{Str15}). 
We see that, when starting from an initial datum connecting a value $u_- < u_+$  (meaning that the time-dependent solution will converge, for large time, towards the {\it increasing} steady state), 
no metastability is observed (see the left picture in Figure \ref{fig3}): 
the stable equilibrium configuration, corresponding to a solution with a horizontal interface located at zero, is in fact attained in a short  time interval.
On the contrary (see the right picture in  Figure \ref{fig3}), when starting from an initial datum connecting boundary values $u_->u_+$,
the corresponding time-dependent solution still develops an internal shock layer on a short time scale, 
but the convergence towards the {\it decreasing} steady state (corresponding to the solution with a vertical interface located at zero) requires much more time: 
for times of the same order as in the previous simulation, the shock layer is still located far from zero (see also the right picture of Figure \ref{fig4}).

\begin{figure}[h]
\centering
\includegraphics[width=6cm,height=5.5cm]{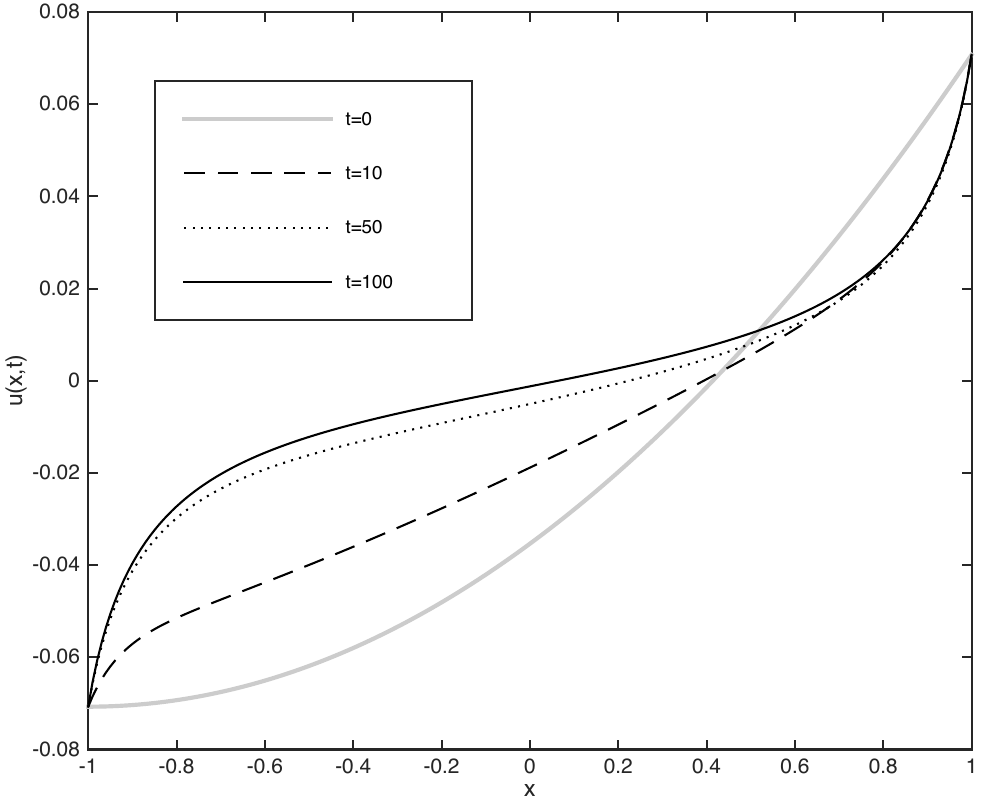}
 \qquad\quad
\includegraphics[width=6.2cm,height=5.5cm]{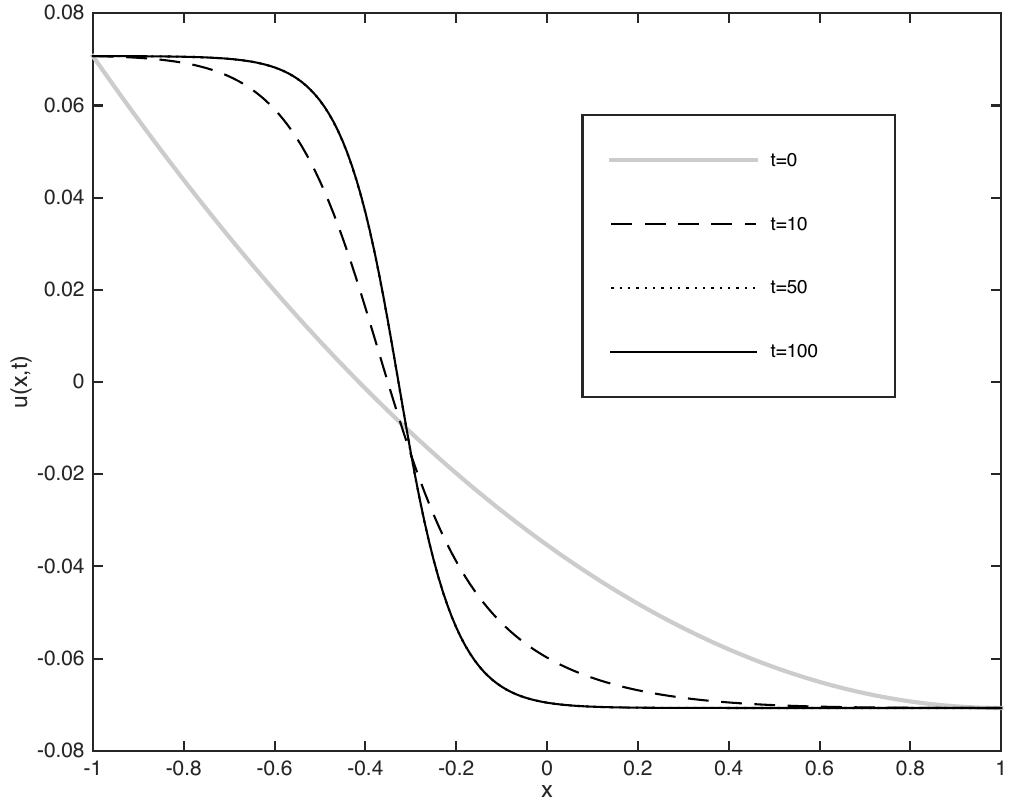}
 \hspace{3mm}
 \caption{\small{The dynamics of the solution to \eqref{burgers2} for $\e=0.005$,  $f(u)=u^2/2$ and $u_0$ increasing with $u_\pm=\pm\sqrt{\e}$ (left) and decreasing with $u_\pm=\mp\sqrt{\e}$ (right).  In both pictures, an interface is formed in a short time. However, in the left picture no metastable behavior is observed as one can see that, for $t=100$, the solution has already reached the steady state. On the opposite, in the right-hand picture the interface  is still very far from zero for times of the same order (the plots for $t=50$ and $t=100$ are  indistinguishable).  }}\label{fig3}
 \end{figure}

Based on these numerical simulations, we thus claim that a necessary condition for the appearance of a metastable behavior under \eqref{ipof} is that
\begin{equation}\label{SC}
u_0(-\ell) > 0 > u_0(\ell).
\end{equation}
Incidentally, we observe that, if the flux function $f$ is {\it concave}, then the necessary condition will be that $u_0(-\ell) < 0 < u_0(\ell)$.
We also observe that condition \eqref{SC} does not require the initial datum to be decreasing; 
however, we see from the numerical simulations that, once \eqref{SC} is satisfied, the solution starting from $u_0$ develops into a {\it decreasing} function in short times, 
and then converges towards the {\it decreasing} steady state (see Figure \ref{fig3new}).
\begin{figure}[h]
\centering
\includegraphics[width=6cm,height=5cm]{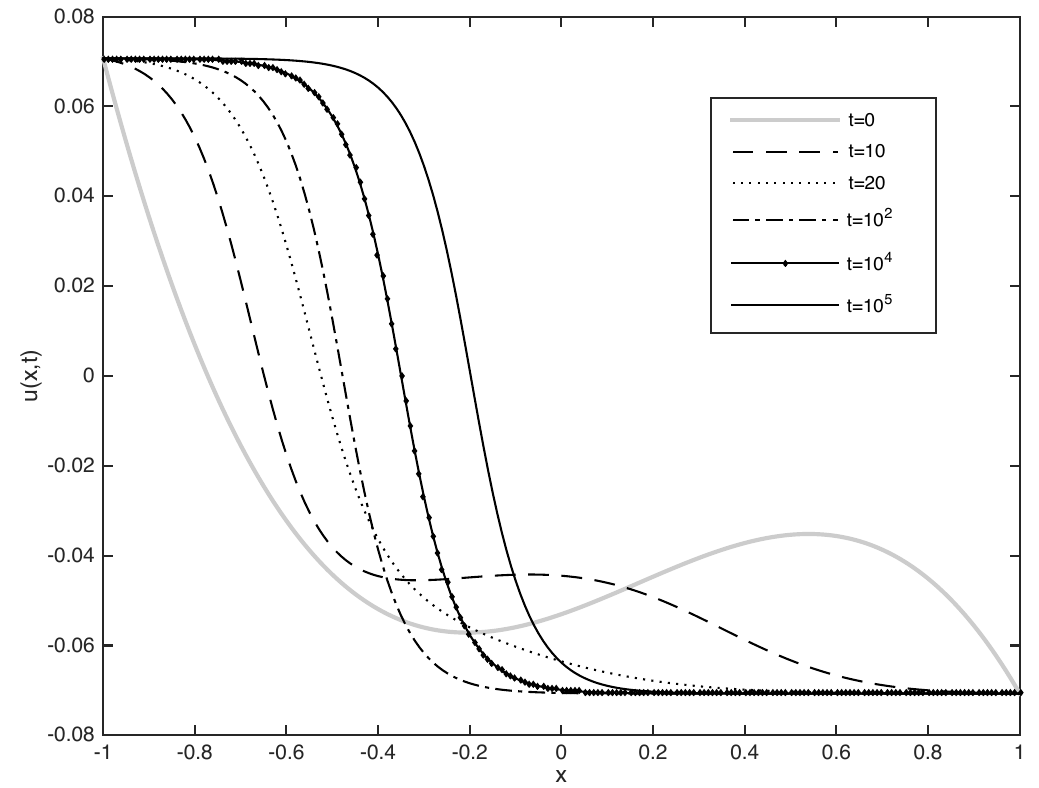}
\qquad \quad
\includegraphics[width=6cm,height=5cm]{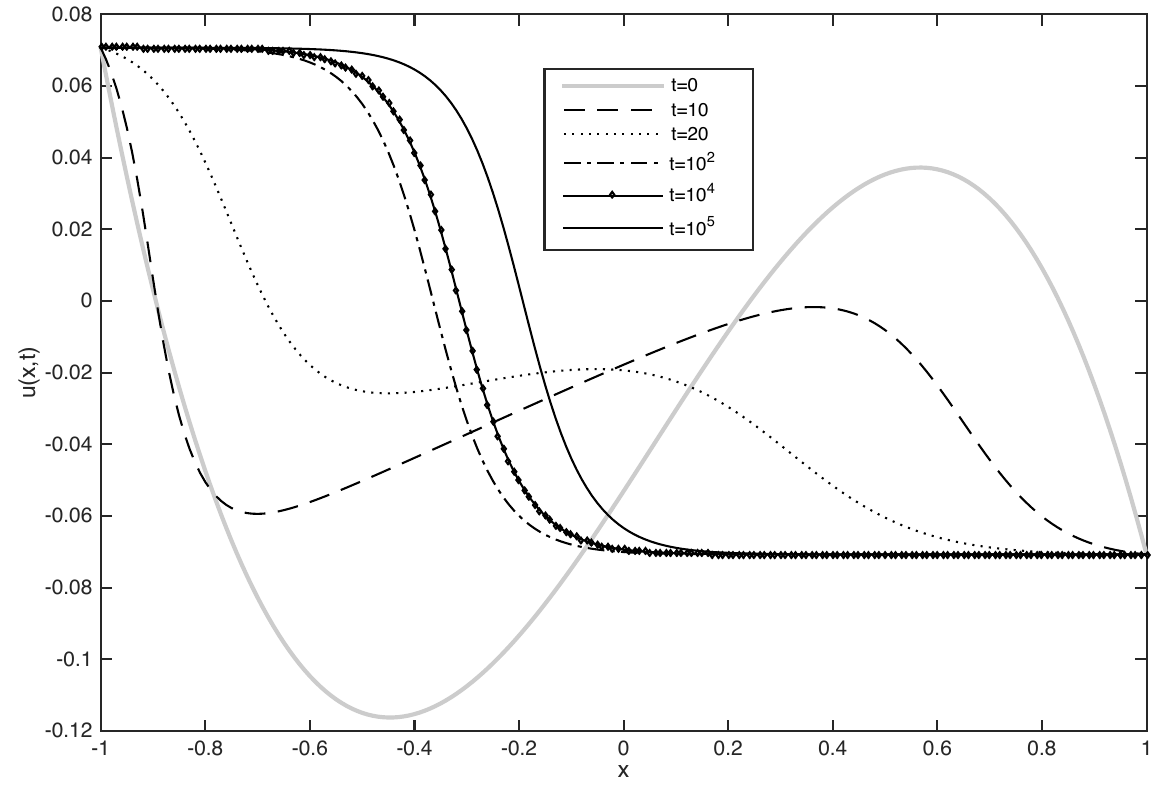}
 \hspace{3mm}
 \caption{\small{The dynamics of the solution to \eqref{burgers2} for $\e=0.005$, $f(u)=u^2/2$ and two different non-monotone initial data (plotted in grey) connecting the values $u_\pm=\mp\sqrt{\e}$. 
As we can see, it is neither necessary for the initial datum to be decreasing nor to be such that $\|u_0\|_{{}_{L^\infty}} \leq |u_\pm|$  to observe a metastable behavior.  }}\label{fig3new}
 \end{figure}

In Figure \ref{fig4a} we plot what happens when the zero of the initial datum is positive; 
we still observe a metastable behavior, but of course the interface will have to move towards the left to reach its asymptotic configuration.
\begin{figure}[ht]
\centering
\includegraphics[width=7cm,height=5.5cm]{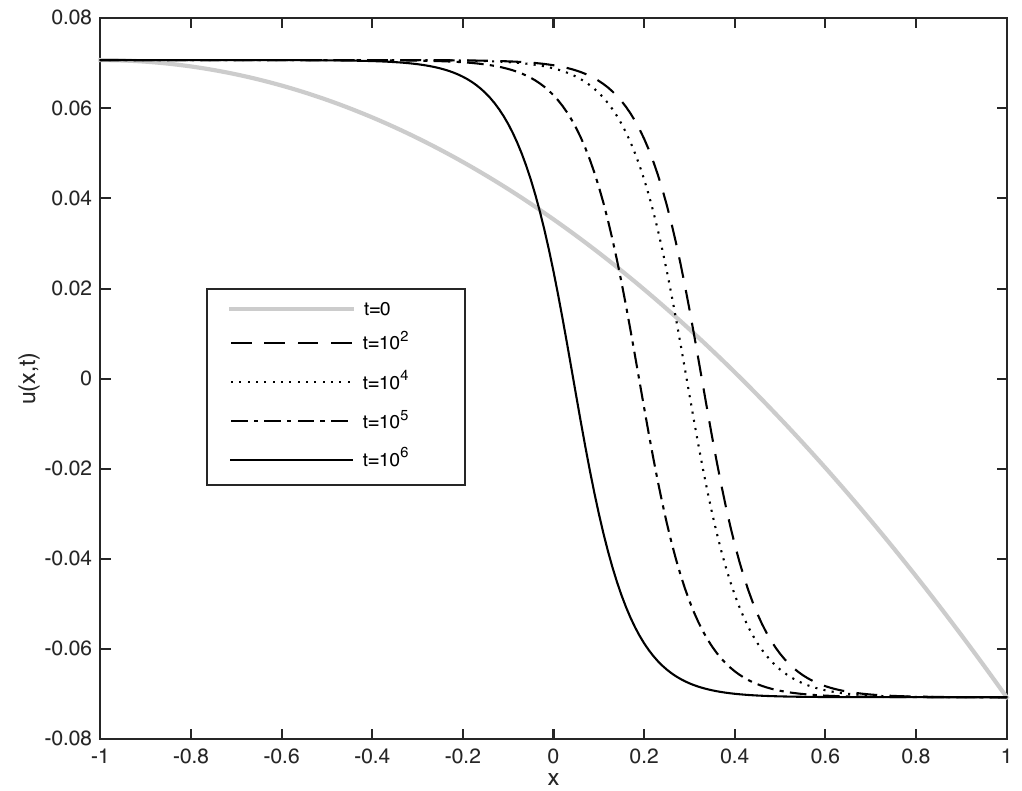}
 \caption{\small{The dynamics of the solution to \eqref{burgers2} for $\e=0.005$, $f(u)=u^2/2$, $u_\pm=\mp\sqrt\e$ and $u_0$ decreasing and such that $u_0(x_0)=0$ for some $x_0>0$. In this case,  the interface is moving with  negative speed.}}\label{fig4a}
 \end{figure}

In Figure \ref{fig6} we show that assumption \eqref{ipof}$_{ii}$ is necessary for the appearance of a metastable behavior: 
when $f(u_-) \neq f(u_+)$ the solution still exists but it does not display a slow convergence towards the equilibrium.
\begin{figure}[ht]
\centering
\includegraphics[width=6.5cm,height=5.5cm]{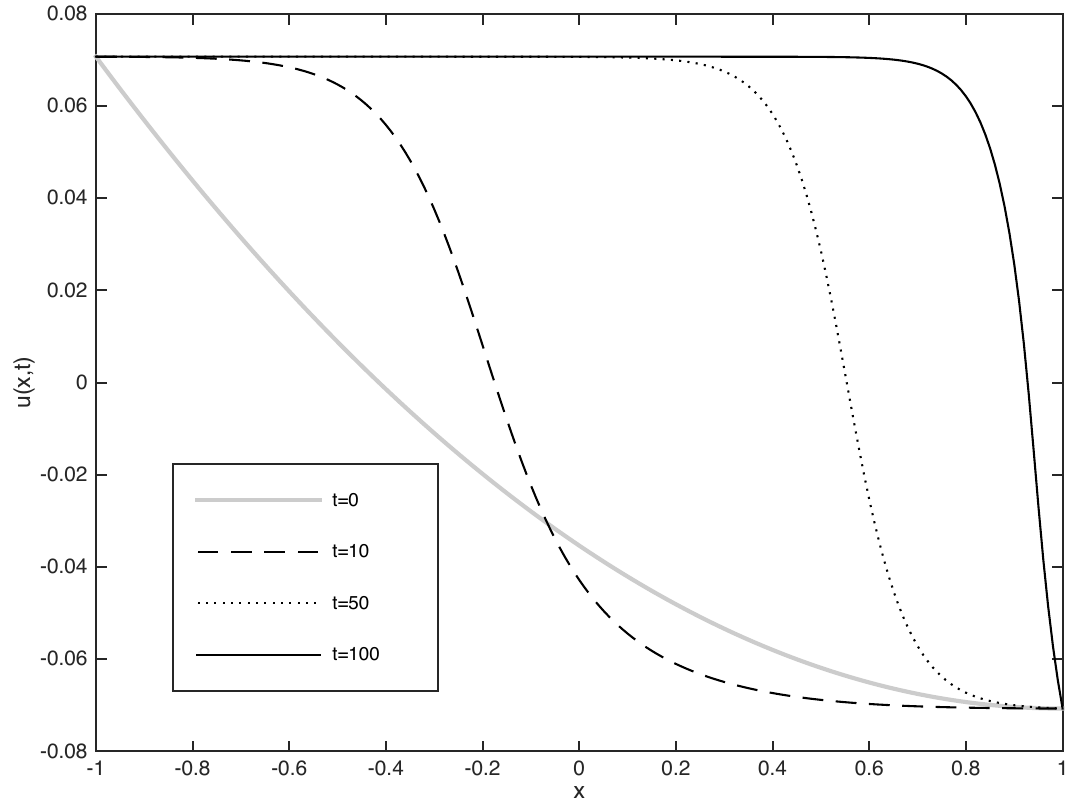}
\qquad \quad
\includegraphics[width=6.5cm,height=5.5cm]{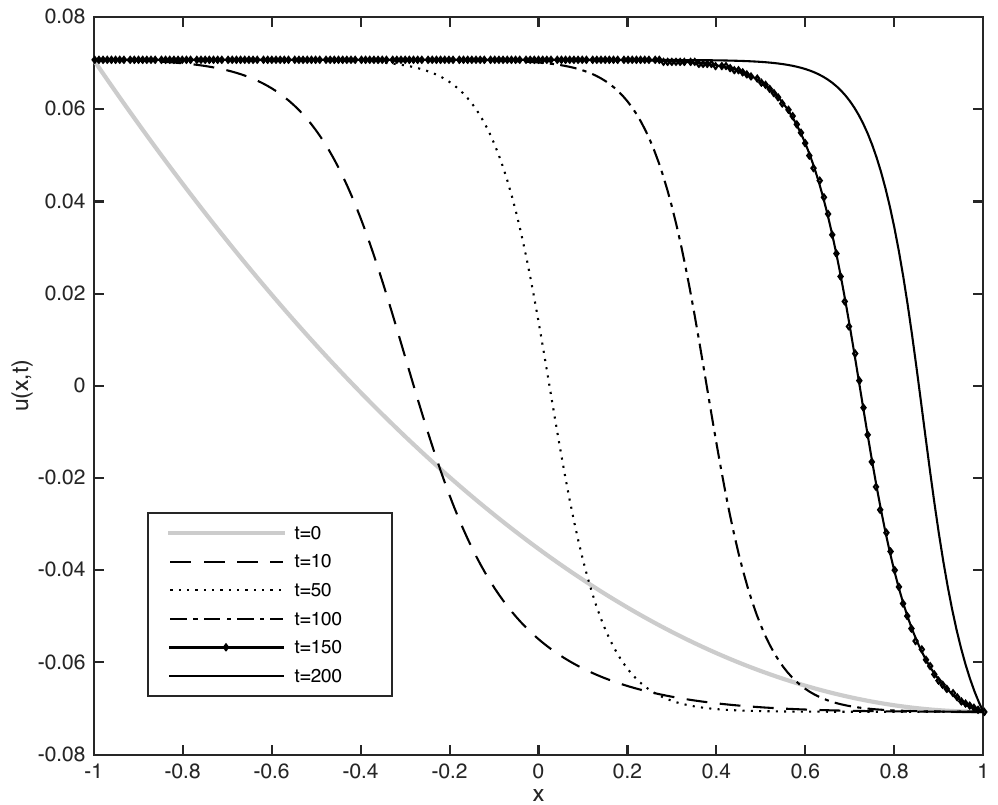}
 \hspace{3mm}
 \caption{\small{The dynamics of the solution to \eqref{burgers2} with $u_\pm = \mp \sqrt{\e}$, initial datum $u_0(x)=\sqrt{\e}\left(\frac{1}{2}x^2-x-\frac{1}{2}\right)$ and $\e=0.005$;  the flux function is $f(u)=\frac{1}{2}(u+a \, \sqrt{\e})^2$ with $a=0.25$ (left) and $a=0.1$ (right), so that $f(u_+)\neq f(u_-)$. We can see that the asymptotic steady state is attained in a short time scale. We also observe that $f(u_-)- f(u_+)= 2 a \e$  and the smaller this difference, the slower the convergence.}}\label{fig6}
 \end{figure}

To have an idea of how the size of the parameter $\e$ influences the speed rate of convergence of the solution towards the steady state, 
Figure \ref{fig4} shows the solution to \eqref{burgers2} for two different values of  $\e$: 
we can clearly see that, for times of the same order $t=5\cdot10^4$, on the one side the solution corresponding to a bigger value of $\e$ has already reached its asymptotic configuration
(corresponding to a solution with an interface located at $x=0$), while on the other side the solution corresponding to a smaller value of $\e$ has an interface still located far from zero. 

\begin{figure}[ht]
\centering
\includegraphics[width=6.5cm,height=5.5cm]{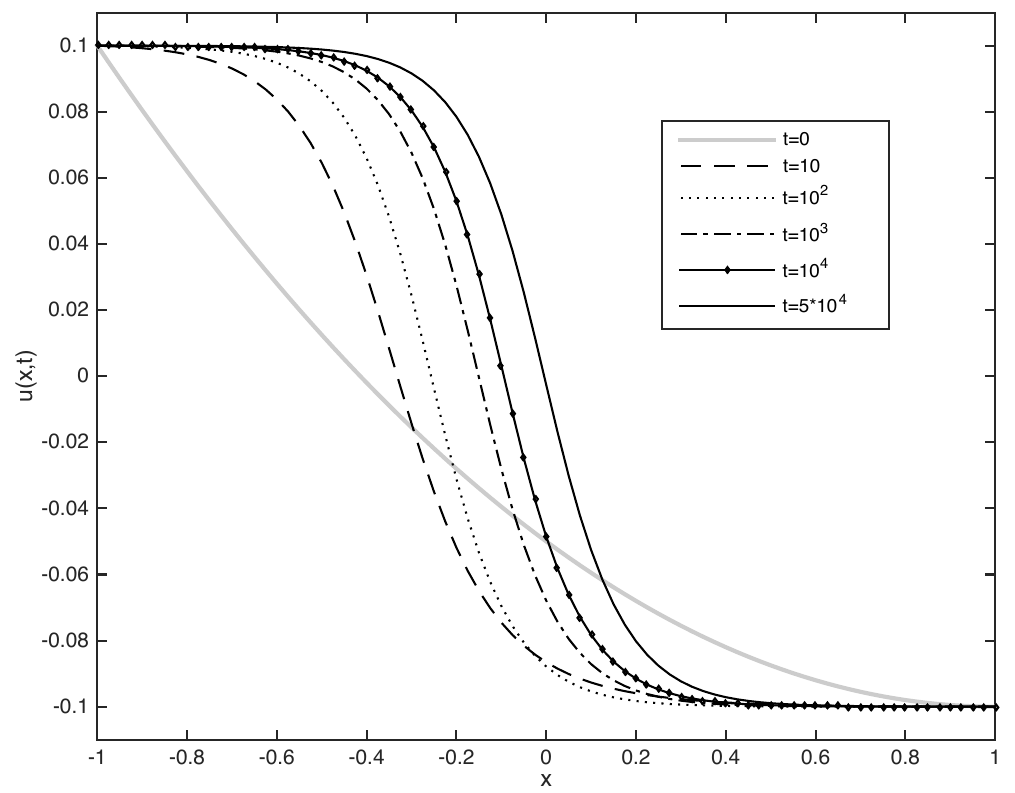}
\qquad\quad
\includegraphics[width=6.5cm,height=5.5cm]{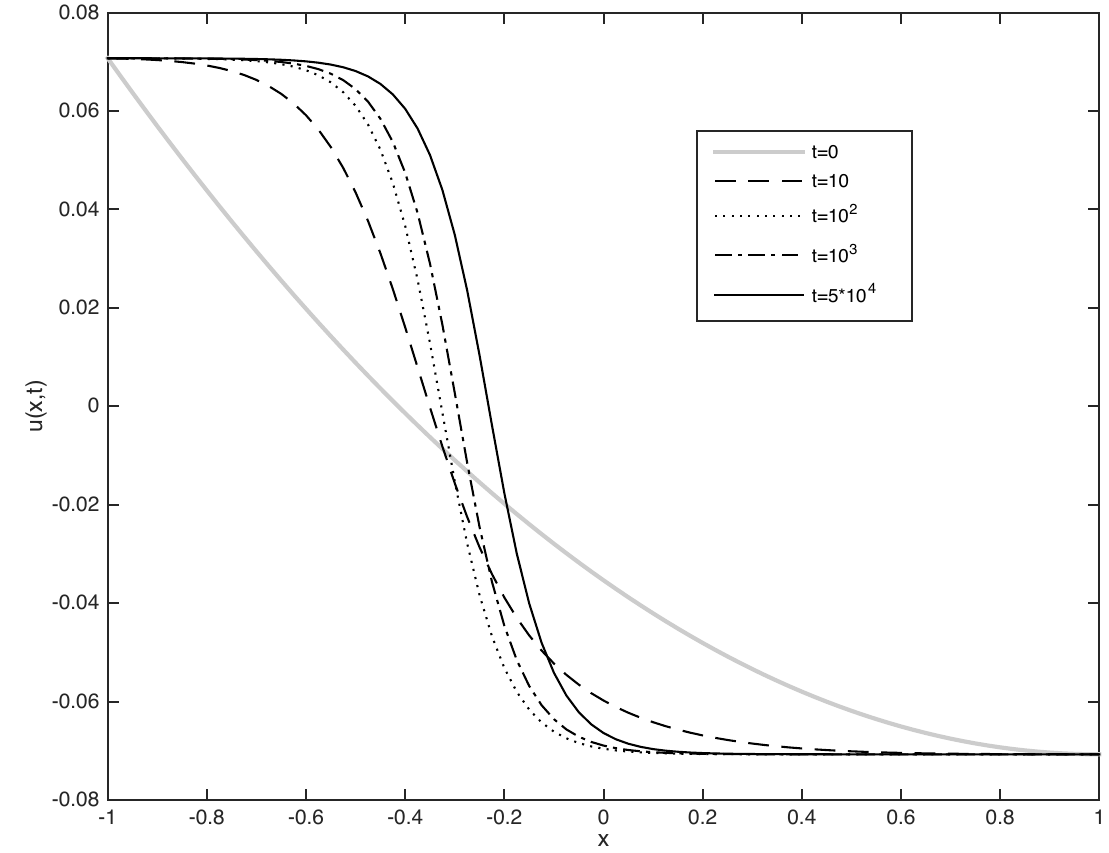}
 \hspace{3mm}
 \caption{\small{The dynamics of the solution to \eqref{burgers2} with $u_\pm = \mp \sqrt{\e}$  and initial datum $u_0(x)=\sqrt{\e}\left(\frac{1}{2}x^2-x-\frac{1}{2}\right)$, for $\e=0.01$ and $\e=0.005$ respectively; 
 in both cases, after an interface located at some point $\xi$ of the interval is formed, it starts to move towards its equilibrium configuration $\xi=0$, 
 but the time needed for the convergence becomes bigger as $\e$ becomes smaller.}}\label{fig4}
 \end{figure}
 
Finally, in Figure \ref{quellabella}, we fix $\e=0.001$ and we see, on the one hand, that the smooth solution reaches its equilibrium configuration only for times of the order $10^{12}$ (left picture); 
on the other hand, in the right picture we show that also starting with a small discontinuous initial datum the (smooth) solution displays the same metastable behavior of the previous cases.
 
 \begin{figure}[ht]
\centering
\includegraphics[width=6.8cm,height=5.5cm]{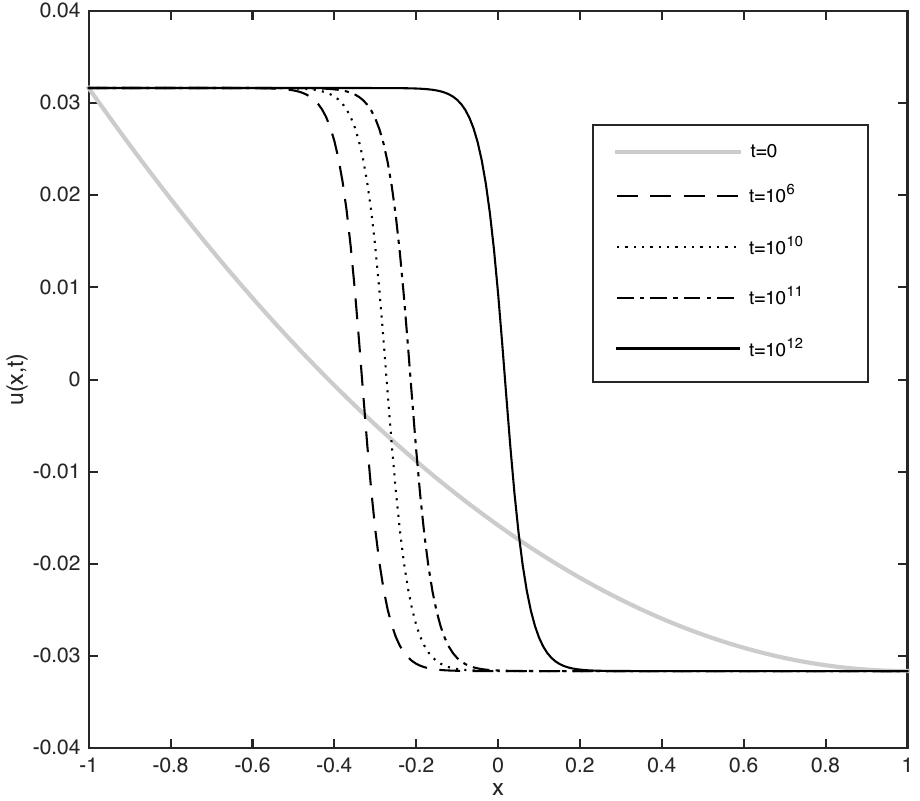}
\qquad\quad
\includegraphics[width=6.8cm,height=5.5cm]{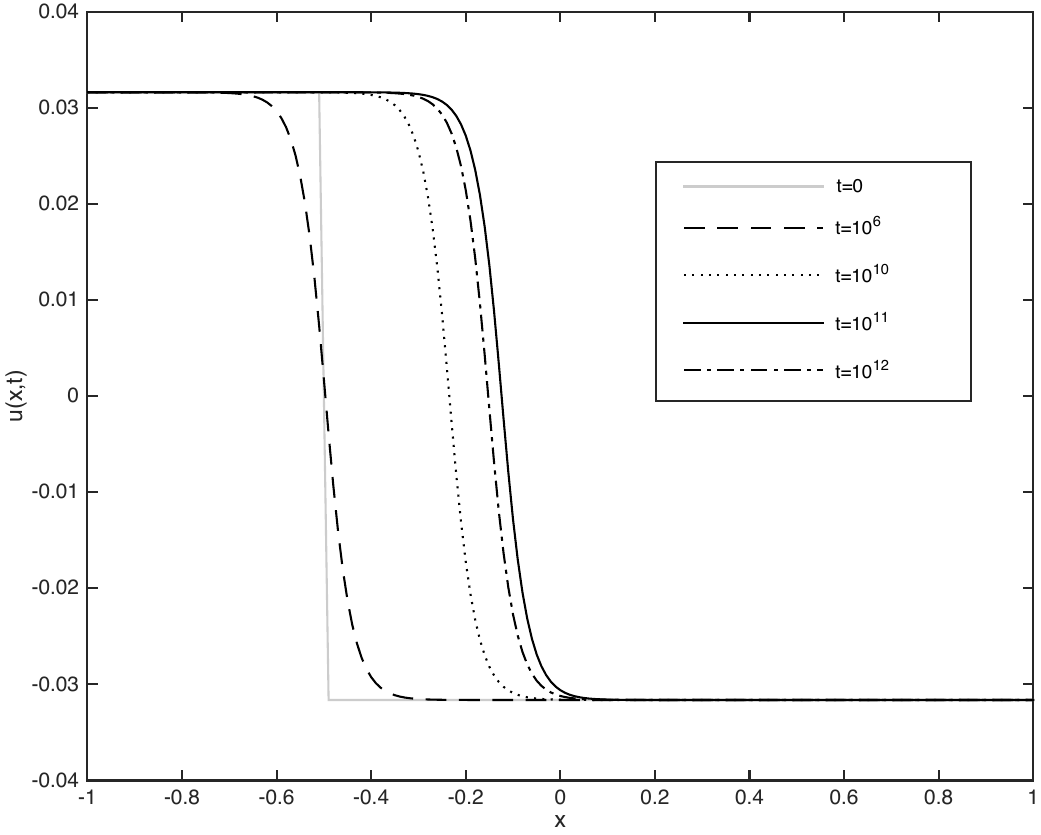}
 \hspace{3mm}
 \caption{\small{Here $\e=0.001$ and the data are as in Figure \ref{fig4} for the left picture, while for the right one we take $u_0(x)=\sqrt{\e}\left(\chi_{(-1,-0.5)}-\chi_{(-0.5,1)}\right)$ .}}\label{quellabella}
 \end{figure}

\section{The metastable dynamics: a rigorous approach}\label{meta}
In this section, we analyze the occurrence of a metastable dynamics for the solutions to \eqref{burgers2} under assumption \eqref{ipof} with $u_- > u_+$ (compare with \eqref{SC}), 
according to what we observed numerically in the previous section. 
To prove the appearance of this pattern, we mean to apply the strategy first developed in \cite{MS} and subsequently extended to general systems in \cite{Str17}; 
it can be divided into three main steps (for more details, see \cite[Section 2]{MS}) that we recall here for the reader's convenience.

\smallbreak
\begin{itemize}
\item {\bf Step I. The family of approximate steady states.} The first step is the construction of a one-parameter family {\it  of approximated steady states} $\{ U^\e(x;\xi) \}_{\xi}$. 
Precisely, given the parameter $\xi \in I$, the generic element $U^\e(x;\xi)$ is built in such a way that the following assumption is satisfied:

\smallbreak
\begin{description}
\item[{\bf H1}] There exists a family of smooth positive functions $\xi \mapsto \Omega^\e(\xi)$ such that $\Omega^\e \to 0$ for $\e \to 0$, uniformly in any compact subset of $I$, and 
\begin{equation*}
	|\langle \psi(\cdot ), \mathcal P^\e[U^\e(\cdot;\xi)]\rangle| \leq \Omega^\e(\xi)|\psi|_{L^\infty}, \quad \forall \, \psi \in C^{\infty}(I), \, \forall \, \xi \in I,
\end{equation*}
where $\mathcal{P}^\e$ is the operator on the right-hand side of \eqref{burgers2}. Moreover, there exists $\bar{\xi} \in I$ such that $\Omega^\e(\bar \xi) \equiv 0$.
\end{description}

\smallbreak
\noindent Assumption {\bf H1} states that each element of the family satisfies the stationary equation up to an error that is small in $\e$ and is measured by $\Omega^\e$, 
while the fact that $\Omega^\e$ vanishes when evaluated in $\bar \xi$ incorporates the property that the specific element $U^\e(x;\bar \xi)$ corresponds to the exact steady state of the equation. 
From now on, we will refer to $\bar \xi$ as the final equilibrium location for the parameter $\xi$, as the convergence of $\xi$ towards $\bar \xi$ will describe the convergence of a solution towards its asymptotic configuration.
 
\smallbreak
 
\item {\bf Step II. Linearization.}
Once the family $\{U^\e(x;\xi)\}_\xi$ is built up, the second step is the linearization of the original system \eqref{NonLBurg} around one of its elements. 
Hence, one has to look for a solution $u$ of the form
\begin{equation}\label{decu}
u(x,t)= U^\e(x;\xi(t))+v(x,t),
\end{equation}
with $\xi=\xi(t)\in I$  and the perturbation ${ v}={v}(x,t)\in L^2(I)$ to be determined. 
The key idea here is the following: in order to describe the dynamics of the solutions up to the formation of the internal interface and throughout their evolution towards the asymptotic limit, 
one supposes the parameter $\xi$ to depend on time, so that its evolution describes the asymptotic convergence of the interface towards the equilibrium. 
Essentially, with the decomposition \eqref{decu}, we reduce the evolution of the solution to the PDE to a one-dimensional dynamics for the parameter $\xi$.  

\smallbreak

\item {\bf Step III. Spectral assumptions.} 
As for the final step of the strategy, the idea is to derive an equation for the perturbation $v$, to be coupled with an equation of motion for the parameter $\xi$. 
To obtain the desired equations, the following assumption describing the distribution of the eigenvalues of the linearized operator around $U^\e$, named here $\mathcal L^\e$, has to be satisfied.

\smallbreak
\begin{description}
\item[{\bf H2}] The linear operator $\mathcal L^\e$ has a discrete spectrum composed by real and semi-simple eigenvalues $\{ \lambda_k^\varepsilon(\xi) \}_{k \in \N}$ such that, 
for any $\xi\in I$, the first eigenvalue $\lambda_1^\varepsilon$ satisfies $$\lim_{\e \to 0}\lambda_1^\varepsilon(\xi) = 0,$$  
while the rest of the spectrum  is negative and bounded away from zero, i.e.
\begin{equation*}
         \lambda^\varepsilon_k(\xi) \leq -C, \qquad \mbox{ for all} \ k \geq 2,
\end{equation*}
for some constant $C>0$ independent of $k$, $\varepsilon$ and $\xi$. 
\end{description}

\smallbreak
Let us note that in assumption {\bf H2} there are no requests on the sign of the first eigenvalue, since what it is crucial is the presence of a spectral gap, encoded in the request {$\lambda^\varepsilon_k(\xi) \leq -C$}.
More precisely, assumption {\bf H2} requires that there exists one eigenvalue that goes to zero as $\e \to 0$ (either positive or negative),
while all the other ones are negative and bounded away from zero. 
This property will be translated into the fact that all the components of the perturbation, except the first one, have a very fast decay in time, 
and the slow motion for the internal interface will only be a consequence of the location of the first eigenvalue. 
Indeed, heuristically, the long time dynamics is described by terms like $e^{\lambda_1^\e t}$, 
so that $\lambda_1^\e$  characterizes the speed rate of convergence of the time-dependent solution towards its equilibrium configuration. 
Hence, the smaller is $\lambda_1^\e$, the slower is the speed and the longer is the time of convergence, as expected.
 \end{itemize}

\subsection{Construction of the family $U^\e$}
In order to apply the strategy just described, we start with the construction of the family of approximated steady states $\{ U^\e\}$. 
There are several possible choices (see, for instance, the MMAE used in \cite{ReynWard95, SunWard99}), 
one of them being to match at a given point $\xi \in I$ two stationary solutions of \eqref{burgers2} satisfying the left and the right boundary conditions, respectively, together with the request $U^\e(\xi;\xi)=0$.
\smallbreak
Precisely, denoted by $U^\e_-$ and $U^\e_+$ the smooth stationary solutions of \eqref{burgers2} in the intervals $(-\ell,\xi)$ and $(\xi, \ell)$ respectively, which satisfy
\begin{equation*}
U^\e_-(-\ell;\xi)=u_-, \quad U^\e_+(\ell;\xi)=u_+ \quad {\rm and } \quad U^\e_-(\xi;\xi)=U^\e_+(\xi;\xi)=0,
\end{equation*}
we define the generic element of the family $\{U^\e\}_{\xi \in I}$ as
\begin{equation}\label{Uapprox}
	U^{\varepsilon}(x;\xi)=\left\{\begin{aligned}
		&U^\e_-(x;\xi) &\qquad & x \in {(-\ell,\xi)} \\
		&U^\e_+ (x;\xi)&\qquad & x \in {(\xi,\ell).}
           \end{aligned}\right.
\end{equation}
In order to show that assumption {\bf H1} is satisfied, we now need to compute $\mathcal P^\e[U^\e]$ with $U^\e$ given as in \eqref{Uapprox}, showing that this term is indeed small with respect to $\e$.

Recalling that the decreasing steady state to \eqref{burgers2} is implicitly given by \eqref{ISS} with $C=-\kappa$, we define 
$$
\Psi(\k,u):= \int_0^{u}  \frac{\sqrt{\e^2-(f(s)-\kappa)^2}}{\k-f(s)} \, ds.
$$
Similar computations as the ones done in {Section \ref{sezSS}} show that
\begin{equation*}
\begin{aligned}
&\Psi(\cdot, u_-)  \ \ \mbox{is decreasing},  \quad &\Psi(f(u_-),u_-)= +\infty, \\
&\Psi(\cdot, u_+) \ \ \mbox{is increasing},  \quad &\Psi(f(u_+),u_+)= -\infty.
\end{aligned}
\end{equation*}
Moreover, $\displaystyle \lim_{\k \to \e} \Psi(\k, u_\pm)= c_\pm$, being $c_\pm \lessgtr 0$ and  $\displaystyle \lim_{\e \to 0}c_\pm =0$ (for example, if $f(u)=u^2/2$, then {$c_{\pm}=\pm c_D/2$}, where $c_D$ has been defined in Theorem \ref{teodecr}). Then, there exist unique $\k_\pm =\k_\pm(\xi) \in (f(u_\pm),\e)$ such that
\begin{equation*}
\Psi(\k_\pm,u_\pm) \pm \ell= \xi,
\end{equation*}
provided $-\ell+c_-<\xi<\ell+c_+$; this is a small (since $c_\pm$ are small in $\e$) restriction on the choice of $\xi$ which implies that $|\xi|\neq\ell$. The corresponding functions $U^\e_\pm$ are implicitly defined by
\begin{equation}\label{Uimplicit}
\Psi(\k_\pm, U^\e_\pm {(x; \xi)}) =\xi-x.
\end{equation}
Because of the construction of the generic element of the family $\{ U^\e \}$, it is easy to check that the error made by $U^\e$ from being the exact steady state of the problem is concentrated in the  gluing point  $x=\xi$; precisely, a straightforward computation shows that
\begin{equation*}
\begin{aligned}
\langle \psi(\cdot ), \mathcal P^\e[U^\e(\cdot;\xi)] \rangle &=  \psi(\xi)(\k_-(\xi)-\k_+(\xi)) \quad \mbox{ for any} \ \ \psi \in C^1(I),
\end{aligned}
\end{equation*}
so that, in distributional sense, it is
\begin{equation*}
\mathcal P^\e[U^\e(\cdot; \xi)]=(\k_-(\xi)-\k_+(\xi)) \delta_{x=\xi},
\end{equation*}
and we  need to evaluate  $\k_--\k_+$ in order to give an expression of $\Omega^\e(\xi)$ as defined in assumption {\bf H1}.
To approximately compute such difference, we observe that, in view of the convexity of the flux function $f$, the following bounds hold: 
\begin{equation}\label{boundF}
\begin{aligned}
&f(u_\pm)+f'(u_+)(u-u_+) \leq f(u) \leq \frac{f(u_\pm)}{u_+} u,  \qquad & u \in [u_+,0],\\
& f(u_\pm)-f'(u_-)(u_- -u) \leq f(u) \leq \frac{f(u_\pm)}{u_-} u, \qquad & u \in [0, u_-]. 
\end{aligned}
\end{equation}
On the one side, $\k_+$ is implicitly defined by
\begin{equation*}
\xi-\ell= \int_0^{u_+} \frac{\sqrt{\e^2-(f(s)-\k_+)^2}}{\k_+-f(s)} \, ds \geq \int_0^{u_+} \frac{\sqrt{\e^2}}{\k_+-f(s)} \, ds
\end{equation*}
{(recall that $u_+$ is negative and $k_+ - f(s) > 0$).}
Hence, by using the upper bound \eqref{boundF} for $u \in [u_+,0]$, we obtain
\begin{equation*}
\begin{aligned}
\frac{\xi-\ell}{\e} &\geq \int_{u_+}^0 \frac{ds}{f(s)-\k_+} \geq \int_{u_+}^0 \frac{ds}{ \frac{f(u_\pm)}{u_+} s-\k_+} = \frac{u_+}{f(u_\pm)}\log \left(\left|\frac{f(u_\pm)}{u_+} s-\k_+\right|\right) \Big|^{0}_{u_+}, \\
\end{aligned}
\end{equation*}
that is,
\begin{equation*}
\begin{aligned}
e^{\frac{f(u_\pm)}{u_+}\frac{\xi-\ell}{\e} }&\leq   \frac{\k_+}{\k_+- f(u_\pm)} \quad  \Longrightarrow \quad \k_+ e^{\frac{f(u_\pm)}{u_+}\frac{\xi-\ell}{\e}} -f(u_\pm)e^{\frac{f(u_\pm)}{u_+}\frac{\xi-\ell}{\e}} \leq \k_+.\\
\end{aligned}
\end{equation*}
In particular, by summing and subtracting $f(u_\pm)$ on the right hand side, the following bound for the difference $\k_+-f(u_\pm)$ holds: 
\begin{equation*}
\k_+-f(u_\pm) \leq \frac{f(u_\pm)}{e^{\frac{f(u_\pm)}{u_+}\frac{\xi-\ell}{\e}}-1}.
\end{equation*}
On the other side, we deduce
\begin{equation*}
\begin{aligned}
\xi-\ell&= \int_0^{u_+} \frac{\sqrt{\e^2-(f(s)-\k_+)^2}}{\k_+-f(s)} \, ds \leq \int_{u_+}^0 \frac{\e-f(s)+\k_+}{f(s)-\k_+} \, ds \\ & \leq u_+ + \e \int_{u_+}^0 \frac{ds}{f(u_\pm)+f'(u_+)(s-u_+)-\k_+} ,
\end{aligned}
\end{equation*}
where, this time, we used the lower bound in \eqref{boundF}. By doing similar computations as above, we end up with
\begin{equation*}
\k_+-f(u_\pm) \geq \frac{u_+f'(u_+)}{e^{f'(u_+)\frac{\xi-\ell-u_+}{\e}}-1}.
\end{equation*}
As concerning $\k_-$, we have
$$
\xi+\ell= \int_0^{u_-} \frac{\sqrt{\e^2-(f(s)-\k_-)^2}}{\k_--f(s)} \, ds \leq \int_0^{u_-} \frac{\e}{\k_--f(s)} \, ds \leq \e\int_{0}^{u_-} \frac{ds}{\k_- -\frac{f(u_\pm)}{u_-} s}, 
$$
and 
\begin{equation*}
\begin{aligned}
\xi+\ell&= \int_0^{u_-} \frac{\sqrt{\e^2-(f(s)-\k_-)^2}}{\k_--f(s)} \, ds \geq \int_0^{u_-} \frac{\e+f(s)-\k_-}{\k_--f(s)} \, ds \\
&\geq -u_-+ \e \int_{0}^{u_-} \frac{ds}{\k_-+f'(u_-)(u_--s)-f(u_\pm)};
\end{aligned}
\end{equation*}
we can thus proceed as before to obtain upper and lower bounds on the difference $\k_- - f(u_\pm)$. In conclusion, collecting all the computations we have
\begin{equation}\label{kpm}
\begin{aligned}
\frac{u_+f'(u_+)}{\exp\left(f'(u_+){(\xi-\ell-u_+)/}{\e}\right)-1}\leq \k_+ - f(u_\pm) &\leq \frac{f(u_\pm)}{\exp\left({f(u_\pm)(\xi-\ell)/{\e \, u_+}}\right)-1}, \\
\frac{u_-f'(u_-)}{\exp\left(f'(u_-){(\xi+\ell+u_-)/}{\e}\right)-1}\leq \k_- - f(u_\pm) & \leq \frac{f(u_\pm)}{\exp\left({f(u_\pm)(\xi+\ell)/{\e \, u_-}}\right)-1}.
\end{aligned}
\end{equation}
\smallbreak
\noindent These bounds show that  the difference $|\k_--\k_+|$ is {\it exponentially small} with respect to $\e$ if we properly choose the boundary conditions; 
indeed, the convergence to zero of $\k_--\k_+$ is dictated by terms like $1/\exp (\mp f'(u_\pm)/\e )$ and $1/\exp (f(u_\pm)/\e u_\pm)$ on both sides. 
Therefore, the error $\Omega^\e$ is exponentially small for $\e\to0$, uniformly in any compact subset of $(-\ell,\ell)$, provided $u_\pm$ are chosen so that
$$\lim_{\e \to 0}\frac{f'(u_\pm)}{\e} = \mp\infty \qquad \mbox{and} \qquad \lim_{\e \to 0} \frac{f(u_\pm)}{\e u_\pm} = \mp\infty.$$
Furthermore, because of the properties of the function $\Psi(\cdot,u)$, the difference 
$
g(\xi):= \k_-(\xi)-\k_+(\xi)
$
is monotone decreasing and such that $\displaystyle\lim_{\xi \to \pm \ell+c_\pm} g(\xi)\lessgtr0$.
As a consequence, there exists a unique value $\bar \xi$ such that $g(\bar \xi)=0$ and  $U^\e(x;\bar \xi)$ is the unique steady state of the system. 
Assumption {\bf H1} is thus satisfied. 

For instance, in the case of a Burgers flux $f(u)=u^2/2$, we have $f'(u_\pm)=u_\pm$; if $u_\pm=\mp\sqrt{\e}$,
there exist positive constants $c_1$ and $ c_2$  such that
\begin{equation}\label{bounderror}
|\Omega^\e(\xi)|= |\k_-(\xi)-\k_+(\xi)| \leq  c_1 \, \e  \,  e^{-c_2(\ell-|\xi|)/ \sqrt{\e}},
\end{equation}
and we have $\Omega^\e(0)=0$, namely $\bar\xi=0$.

Let us stress that the choice of the boundary data is fundamental in order to obtain an exponentially small error $\Omega^\e$; 
indeed, if we chose $\vert u_\pm \vert=\e$ (rather than $\sqrt{\e}$), the error would only be algebraically small in $\e$, being its convergence to zero dictated by  $u_\pm f'(u_\pm)$ and $f(u_\pm)$.

\subsection{Linearization and spectral analysis}
In order to linearize the original equation, consider the decomposition \eqref{decu}, where we are taking $U^\e$ smooth 
(for instance, $U^\e$ can be chosen as a smoothed version of \eqref{Uapprox}, see Remark \ref{approssimataU} below). 
Inserting \eqref{decu} into \eqref{burgers2} and recalling that $\xi(t)$ depends on time, we end up with the following PDE for the perturbation $v$:
\begin{equation}\label{eqv}
\d_t v = \mathcal L^\e v + \mathcal P^\e[U^\e]- \d_{\xi} U^\e \frac{d\xi}{dt} + \mathcal R[v,\xi],
\end{equation}
where $\mathcal L^\e v$ is the linearized operator obtained after the linearization \eqref{decu}, while $\mathcal R[v,\xi]$ collects all the higher order terms in $v$.

\smallbreak
Precisely, a straightforward computation shows that
\begin{equation}\label{linoperator}
\mathcal L^\e v = \frac{\e\d_x^2v}{\left(1+ {\d_x U^\e}^2\right)^{3/2}}- \frac{3 \e \d_x^2U^\e \d_xU^\e }{\left(1+ {\d_x U^\e}^2\right)^{5/2}} \d_x v-\d_x (f'(U^\e) v),
\end{equation}
and the equation for $v$ reads
$$
\d_t v = \mathcal L^\e v + \e\,\d_x \left( \frac{\d_xU^\e}{\sqrt{\left(1+ {\d_x U^\e}^2\right)}}\right)- \d_x f(U^\e)- \d_{\xi} U^\e \frac{d\xi}{dt}+ \mathcal R[v,\xi],
$$
which is exactly \eqref{eqv}, being
\begin{equation*}
 \mathcal P^\e[U^\e]=\e \, \d_x \left( \frac{\d_xU^\e}{\sqrt{\left(1+ {\d_x U^\e}^2\right)}}\right)- \d_x f(U^\e),
\end{equation*}
that is small in $\e$ according to assumption {\bf H1}.

\smallbreak
We now mean to verify assumption {\bf H2}; we thus have to exploit spectral properties of the linearized operator $\mathcal L^\e$ defined in \eqref{linoperator}.
Henceforth, for the sake of simplicity, we consider a flux function of Burgers type, i.e., $$f(u)=\frac{u^2}2.$$
We recall that, in this case, $u_\pm=\mp u_*$ for some $u_*\in(0,\sqrt{2\e})$, in view of \eqref{ipodatibordo} and \eqref{ipof}.

\smallbreak

Let us thus rewrite  the linearized operator $\mathcal L^\e$ given in \eqref{linoperator} as
\begin{equation}\label{linopSL}
	\mathcal L^\e v =p(x)\d_x^2 v+ q(x) \d_x v+r(x) v,
\end{equation}
where
\begin{equation}\label{defpqr}
	\begin{aligned}
	p(x):= \frac{\e}{\left(1+ {\d_x U^\e}^2\right)^{3/2}}, \quad
	q(x):=-\frac{3 \e \d_x^2U^\e \d_xU^\e }{\left(1+ {\d_x U^\e}^2\right)^{5/2}}-U^\e \quad \mbox{and} \quad
	r(x):= - \d_xU^\e.
	\end{aligned}
\end{equation}
A straightforward computation shows that, if defining the weight $\rho$ as
\begin{equation*}
	\rho(x):= \exp\left(-\frac{1}{\e}\int_{x_0}^x a^\e(y) \, dy\right), \quad a^\e(x)= U^\e(1+{\d_xU^\e}^2)^{3/2},
\end{equation*} 
the following identity holds:
\begin{equation}\label{defdiLrho}
	(\rho\mathcal  L^\e )v:=\mathcal L^\e_\rho v = \d_x\left(\rho(x)p(x)\d_xv\right)+ \rho(x)r(x) v.
\end{equation}
Going further, for $v,w \in H^1(I)$ it holds that
\begin{equation}\label{Lagrangeidentity}
	\begin{aligned}
	\langle v, \mathcal L^\e w\rangle_{\rho} &= \int_{I} v \left( p(x)\d_x^2w+q(x)\d_x w+r(x)w\right) \, \rho \, dx \\
	&=\int_{I} v \left[\d_x\left(\rho \,  p(x) \d_xw\right)+r(x)w\rho\right] \, dx \\
	&=-\int_{I} \left[\d_xv \left(\rho \,  p(x) \d_xw\right)+r(x)w v\rho \right]\, dx  + v \rho \, p(x) \d_xw \big|^{\ell}_{-\ell} \\
	&=\int_{I} \left[\d_x\left(\d_xv \rho \,  p(x)\right) w+r(x)w v\rho \right]\, dx -\d_xv\rho \, p(x) w\big|^{\ell}_{-\ell} + v \rho \, p(x) \d_xw \big|^{\ell}_{-\ell}\\
	&=\langle\mathcal L^\e v,  w\rangle_{\rho}+ v \rho \, p\d_x w \big|^{\ell}_{-\ell}-\d_xv\rho \, p w\big|^{\ell}_{-\ell},
	\end{aligned}
\end{equation}
where $\langle u , v\rangle_{\rho}= \int_{I} u v \rho \, dx$ is the scalar product in the weighted space $L^2_\rho(I)$. 
The operator $\mathcal L^\e$ is thus formally self-adjoint in $L^2_\rho(I)$,  according to the definition given in \cite[Section 2.3]{SL}.
Formally self-adjoint operators in $L^2_\rho(I)$ satisfy the following statement, coming from the classical Sturm-Liouville theory (cf., e.g., \cite[Section 2.4]{SL}).
\begin{proposition}\label{SL}
Let $\mathcal L$ be a formally self-adjoint operator on $L^2_\rho(I)$ having the form 
$$
	\mathcal{L} v=\partial_x (\rho(x) p(x) \partial_x v)+ \rho(x) r(x) v,
$$
with $p,r, \rho \in C([-\ell,\ell])$ such that $p,\rho >0$ in $[-\ell,\ell]$. Then, the problem
$$
	\mathcal Lv =\lambda \rho v
$$
has an infinite sequence of real eigenvalues $\{\lambda_k\}_{k \in \N}$, such that
\begin{equation*}
	\dots < \lambda_3 < \lambda_2 < \lambda_1, \quad \mbox{and} \quad \lambda_k \to -\infty \ \mbox{as} \ k \to \infty;
\end{equation*}
moreover, 
\begin{equation}\label{UB}
	\lambda_1 \leq \max_{x \in I} r(x).
\end{equation}
Finally, for each eigenvalue $\lambda_k$ there exists a single corresponding eigenfunction $\varphi_k$, having exactly $k-1$ zeros in $I$, 
and the set of the eigenfunctions is an orthonormal basis in the weighted space $L^2_\rho(I)$, i.e.,
\begin{equation*}
	\langle \varphi_i, \varphi_j \rangle_{\rho} = \delta_{ij}.
\end{equation*}
\end{proposition}
\begin{remark}\label{approssimataU}
{\rm
To make Proposition \ref{SL} applicable to the operator $\mathcal{L}^\e$ defined in \eqref{linopSL},
we observe that the function $U^\e$ constructed in \eqref{Uapprox} is an $H^1$-function with a continuous derivative up to the jump located at $x=\xi$, 
and $C^k$ is dense in $H^1$ for arbitrarily large $k$; therefore, we can approximate $U^\e$ with a smooth function up to an arbitrarily small error. 
Henceforth, we will thus work with a smooth approximation of $U^\e$, still denoted by $U^\e$, and apply Proposition \ref{SL}.
}
\end{remark}
We now show that we can actually improve the upper bound \eqref{UB}; 
precisely, we aim at proving that $\sigma(\mathcal L^\e) \subset (-\infty,0)$, that is all the eigenvalues of $\mathcal L^\e$ are {\it negative}. 
To this aim, we need  the following preliminary lemma.

\begin{lemma}\label{lemmaUx}
Let $U$ be the exact stationary solution to \eqref{burgers2}; 
then $\d_x U$ satisfies
\begin{equation*}
	L (\d_x U) =0, \quad Lv:=\bar p(x)\d_x^2 v+ \bar q(x) \d_x v+\bar r(x) v,
\end{equation*}
where $\bar p, \bar q$ and $\bar r$ are defined as in \eqref{defpqr}, replacing   $U^\e$ by $U$.
\end{lemma}
\begin{proof}
Recalling that $U$ solves
\begin{equation*}
	\frac{\e \d_x^2U}{(1+\d_x U^2)^{3/2} }= U \d_xU,
\end{equation*}
{namely $\bar p(x) \d_x^2U= U\d_xU$},
we have
\begin{equation*}
	\begin{aligned}
	\d_x\left(\rho \, \bar p(x) \d_x^2U\right) &+ \rho \bar r(x) \d_xU  \\
	&=\d_x\left(\rho \, \bar p(x) \d_x^2U+ \rho \bar r (x) U\right)- \d_x(\rho\bar r(x)) U  \\
	&=- \d_x(\rho \bar r(x)) U,
	\end{aligned}
\end{equation*}
where in the last equality we used the explicit form of $\bar r$. 
Finally
\begin{equation*}
	\begin{aligned}
	-\d_x(\rho \bar r) &= \rho \d_x^2U+\d_x\rho \d_xU \\
	&=\rho \left[ \d_x^2U-\frac{1}{\e} U\d_x U(1+\d_xU^2)^{3/2}\right]=0, \\
	\end{aligned}
\end{equation*}
since $U$ is the stationary solution.
\end{proof}

\begin{remark}\label{remlemmaUx}{\rm
When considering an element of the family of {\it approximate steady states}, 
formula \eqref{bounderror} states that $U^\e$ solves the stationary equation up to an error that is exponentially small in $\e$, having chosen $u_*=\sqrt\e$. 
Hence, by performing the same computations of Lemma \ref{lemmaUx} with $\mathcal L^\e$ defined as in \eqref{linopSL}-\eqref{defpqr}, 
we can state that $\mathcal L^\e( \d_x U^\e)$ is exponentially small in $\e$, that is, 
the first eigenfunction $\varphi^\e_1$ of $\mathcal L^\e$ relative to the eigenvalue $\lambda^\e_1$ is approximately given (up to its sign) by $\d_xU^\e$.

}
\end{remark}

\begin{proposition}
The eigenvalues $\{\lambda^\e_k\}_{k \in \N}$ of $\mathcal L^\e$ are negative.
\end{proposition}

\begin{proof}
Let us integrate the relation $\mathcal L^\e\varphi^\e_1= \lambda^\e_1 \varphi^\e_1$, with $\mathcal L^\e$ defined as in \eqref{linoperator} or, in shortest notation,
\begin{equation*}
	\mathcal L^\e v:= p(x) \d_x^2 v+ q_1(x) \d_x v + \d_x (U^\e v).
\end{equation*}
By Proposition \ref{SL}, we can assume without loss of generality $\varphi^\e_1 >0$ in $I$. 
Observing that $p'(x)=q_1(x)$, we obtain
\begin{equation*}
	\begin{aligned}
	\lambda^\e_1 \int_I \varphi^\e_1 \, dx &= \int_Ip\,{\varphi^\e_1}'' \, dx+ \int_I q_1 {\varphi^\e_1}' \, dx- \int_I (f'(U^\e)\varphi^\e_1)' \, dx \\
	&=p\,{\varphi^\e_1}'\big|^{\ell}_{-\ell} - \int_I p'\, {\varphi^\e_1}' \, dx+ \int_I q_1 {\varphi^\e_1}' \, dx-(U^\e \varphi^\e_1)\big|^{\ell}_{-\ell} \\
	&=p(\ell){\varphi^\e_1}'(\ell)-p(-\ell){\varphi^\e_1}'(-\ell),
	\end{aligned}
\end{equation*}
where in the last line we used the fact that $\varphi^\e_1(\pm \ell)=0$. 
Since $\int_I \varphi^\e_1 >0$, $p >0$ and ${\varphi^\e_1}'(\pm\ell)\lessgtr0$, for some positive constant $C>0$ we have that
\begin{equation*}
	C\lambda^\e_1 <0,
\end{equation*}
and the proof is completed recalling that $\{ \lambda^\e_k\}_{k \in \N}$ is a decreasing sequence.
\end{proof}

\subsection{Asymptotics for the first and second eigenvalues}

\smallbreak
We first give an estimate of $\lambda_1^\e$, obtained by
 applying the identity \eqref{Lagrangeidentity} with $w=\varphi^\e_1 > 0$ and $v=1/\rho$.  
 We get
\begin{equation}\label{roughlambda1}
	\begin{aligned}
	\lambda_1^\e\int_I\varphi^\e_1\,dx  &= \langle\mathcal L^\e\left(1/\rho\right),  \varphi^\e_1\rangle_\rho+  p\, {\varphi^\e_1}' \big|^{\ell}_{-\ell}+ \frac{\rho_x}{\rho} p\,\varphi^\e_1\big|^{\ell}_{-\ell} \\
	&=p\,{ \varphi^\e_1}' \big|^{\ell}_{-\ell},
	\end{aligned}
\end{equation}
where we used $\mathcal L^\e (1/\rho) =0$ and the fact that $\varphi^\e_1(\pm\ell)=0$. 
Moreover, by differentiating the implicit expressions of $U^\e_{\pm}$  given in \eqref{Uimplicit}, we deduce
\begin{equation*}
	\d_x U^\e_{\pm}= \frac{f(U^\e_{\pm})-\k_{\pm}}{\sqrt{\e^2-(f(U^\e_{\pm})-\k_{\pm})^2}}
\end{equation*}
and this, together with  \eqref{kpm}, leads to
\begin{equation*}
	\d_x U^\e(\pm \ell) = \frac{ f(u_\pm)-{\k_\pm}}{\sqrt{\e^2 -(f(u_\pm)-\k_\pm)^2}} \approx   -\frac{u_*^2}{\e} \, e^{-u_*(\ell-|\xi|)/{\e}}.
\end{equation*}
Finally, from Remark \ref{remlemmaUx}, we can state that $\int \varphi^\e_1 \, dx\approx  u_*$ and
\begin{equation*}
	\begin{aligned}
	{\varphi^\e_1}'(\pm\ell) \approx - \e^{-1}\,  U^\e(\pm\ell) \d_x U^\e(\pm \ell) (1+ {\d_xU^\e(\pm \ell)}^2)^{3/2} \approx \mp u_* ^3 \e^{-2}e^{-u_*(\ell-|\xi|)/{\e}},
	\end{aligned}
\end{equation*}
so that, recalling that $p(x) \approx \e$, from \eqref{roughlambda1} we  have
\begin{equation}\label{rough}
	\lambda_1^\e(\xi)\approx-u_*^2 \e^{-1}e^{-u_*(\ell-|\xi|)/\e},
\end{equation}
In particular, if $u_*=\sqrt{\e}$, then 
$$\lambda_1^\e(\xi)\approx-\,e^{-(\ell-|\xi|)/\sqrt\e}.$$

\smallbreak
\noindent As we already remarked, the large time behavior of the solution is dictated by terms of the order $e^{\lambda^\e_1 t}$; 
hence, as in the linear case, we expect  $\lambda_1^\e$ to give a good approximation of  the speed rate of convergence of the solution towards the asymptotic configuration.

This guess is somehow confirmed by numerical simulations. 
Table \ref{Table1} shows a numerical computation for the location of the shock layer for different values of the parameter $\e$ and $f(u) = u^2/2$. 
The initial datum for the function $u$ is $u_0(x) = u_*\left(\frac{1}{2}x^2 - x - \frac{1}{2}\right) $, being $u_*=\sqrt{\e}$. 
We can clearly see that the convergence to $\bar \xi =0$ is slower as $\e$ becomes smaller. 
\begin{table}[h!]
\begin{center}
\begin{tabular}{|c|c|c|c|}
\hline TIME $t$ &   $\xi(t)$, $\varepsilon=0.03$  &   $\xi(t)$, $\varepsilon=0.01$ &  $\xi(t)$, $\varepsilon=0.005$ 
\\ \hline
\hline $10$ & $-0.2952$ &$-0.3317$  &$-0.3507$  \\
\hline  $10^2$ & $-0.1607$ & $-0.3132$  &$-0.3272$  \\
\hline $10^3$ & $-0.0014$ &$-0.2555$ &$-0.3234$   \\
\hline $5\cdot10^3$ & $-0.0018*10^{-5}$ &$-0.1489$  &$-0.3086$  \\
\hline $10^4$ & $-0.0005*10^{-5}$ &$-0.0925$ &$-0.2936$ \\
\hline $5\cdot10^4$ & $0$ &$-0.0041$ &$-0.2283$  \\
\hline $10^5$ & $0$ &$-0.0009*10^{-1}$ &$-0.1887$  \\
\hline $10^6$ & $0$ &$0$ &$-0.0434$  \\
\hline
\end{tabular}
\caption{\small{The numerical location of the shock layer $\xi(t)$ for \eqref{burgers2}, for different values of the parameter $\varepsilon$.}}\label{Table1}
\end{center}
\end{table}

In Table \ref{Table2}, we use the previous data to compute the (average) speed of $\xi$, and we compare this result with $  e^{-1/\sqrt{\e}}$: 
we notice the resemblance between the two values, whatever the choice of $\e$.

\begin{table}[h!]
\begin{center}
\begin{tabular}{|c|c|c|c|c|}
\hline  SPEED  &   $\varepsilon=0.03$  &    $\varepsilon=0.01$ &   $\varepsilon=0.005$   \\
 \hline
\hline $\Delta x / \Delta t$ & $1.7*10^{-4}$ &$3*10^{-5}$ &$2.8*10^{-7}$   \\
\hline $ e^{-1/\sqrt{\e}}$ & $0.3*10^{-4}$ &$4*10^{-5}$ &$0.7*10^{-7}$   \\
\hline 
\end{tabular}
\caption{\small{In this table, we compute   $\Delta t:= t_{{}_{\rm{F}}}-t_{{}_{\rm{I}}}$, 
being $t_{{}_{\rm{F}}}$ the time where the interfaces reaches the $x$-value zero and $t_{{}_{\rm{I}}}=100$; of course the value of $t_{{}_{\rm{F}}}$ is different in the three cases. 
The value of $\Delta x$ is computed accordingly, by taking the $x$-values corresponding to these times. }}
\label{Table2}
\end{center}
\end{table}

For comparison, in the following  we numerically compute the location of the shock layer for the {\it linear} equation
\begin{equation}\label{linburg}
\d_t u = \e \d_x^2u - \d_xf( u).
\end{equation} 
In this case, it has been proven in \cite{MS} that the speed rate of convergence of the interface towards the equilibrium is indeed {\it exponentially small} with respect to $\e$; precisely, the authors give the following asymptotic expression for the first eigenvalue $\lambda^\e_1$
$$
\lambda_1^\e \approx%
-\frac{1}{\e} \left( \frac{1}{f'(u_-)}-\frac{1}{f'(u_+)}\right)^{-1} \left( -f'(u_+)e^{f'(u_+)(\ell-\xi)/\e}+f'(u_-)e^{-f'(u_-)(\ell+\xi)/\e}\right),
$$
which, in the case $f(u)=u^2/2$ and $u_{\pm}=\mp u_* = \mp \sqrt{\e}$, gives  $\lambda_1^\e(\xi) \approx e^{-(\ell-|\xi|)/\sqrt\e}$.

Table \ref{Table3} shows the  numerical location of the shock 
layer when considering equation \eqref{linburg}. The initial datum is as in the previous simulations, that is, $u_0(x) = u_*\left(\frac{1}{2}x^2 - x - \frac{1}{2}\right) $ with $u_*=\sqrt{\e}$. We can see that the resemblance with the previous data is significant.
\begin{table}[h!]
\begin{center}
\begin{tabular}{|c|c|c|c|c|cl}
\hline TIME $t$ &   $\xi(t)$, $\varepsilon=0.03$  &   $\xi(t)$, $\varepsilon=0.01$ &  $\xi(t)$, $\varepsilon=0.005$ 
\\ \hline
\hline $10$ & $-0.2915$ &$-0.3307$  &$-0.3499$  \\
\hline  $10^2$ & $-0.1435$ & $-0.3116$  &$-0.3271$  \\
\hline $10^3$ & $-0.0005$ &$-0.2467$ &$-0.3226$   \\
\hline $5\cdot10^3$ & $-0.0025*10^{-5}$ &$-0.1348$  &$-0.3055$  \\
\hline $10^4$ & $-0.0005*10^{-5}$ &$-0.0786$ &$-0.2887$ \\
\hline $5\cdot10^4$ & $0$ &$-0.0020$ &$-0.2188$  \\
\hline $10^5$ & $0$ &$-0.0002*10^{-1}$ &$-0.1779$  \\
\hline $10^6$ & $0$  &$0$ & $-0.0342$ \\
\hline
\end{tabular}
\caption{\small{The numerical location of the shock layer $\xi(t)$ for equation \eqref{linburg}, for different values of the parameter $\varepsilon$.}}
\label{Table3}
\end{center}
\end{table}

Based on these numerical simulations, we thus expect that the estimate \eqref{rough} gives a good qualitative approximation of the order of the first eigenvalue of the linearized operator. 
\smallbreak
In order to give a bound for the second (and subsequent) eigenvalue $\lambda_2^\e$, we follow the approach of \cite{FLMS} 
where the authors approximate $\psi_2^\e$, the second eigenfunction of the adjoint operator $\mathcal L^{\e,*}$, with the second eigenfunction of the operator 
\begin{equation}\label{opiperbolico}
\mathcal N^{\e,*} v := \e \d_x^2 v \pm u_* \d_x v,
\end{equation}
obtained from $\mathcal L^{\e,*}$ by approximating $U^\e$ with 
 \begin{equation*}
	U^0(x;\xi)=u_*\chi_{{}_{(-\ell,\xi)}}(x)-u_*\chi_{{}_{(\xi,\ell)}}(x).
\end{equation*}
In particular, they give the following asymptotic expression for the second eigenfunction $\psi_2^\e$ (for more details, see \cite[Section 4]{FLMS})
\begin{equation*}
	\psi_2^\e(x) \approx \psi_2^0(x)=\begin{cases}
			c_-e^{-u_\ast(x+\ell)/2\e}\sin\left(\sqrt{-4\e\lambda_2^\e-u_\ast^2}(x+\ell)/2\e\right) \qquad & x\leq\xi \\
			c_+e^{u_\ast(x-\ell)/2\e}\sin\left(\sqrt{-4\e\lambda^\e_2-u_\ast^2}(x-\ell)/2\e\right)  & x>\xi,
		\end{cases}
\end{equation*}
where 
\begin{align*}
	c_\pm & = e^{-u_\ast(\xi\pm\ell)/2\e}\sin\left(\sqrt{-4\e\lambda_2^\e-u_\ast^2}(\xi\pm\ell)/2\e\right).
\end{align*}
Hence, $\psi_2^\e$ is defined if and only if
\begin{equation*}
-4\e\lambda^\e_2-u_\ast^2 \geq0 \qquad \Rightarrow \qquad \lambda^\e_2 \leq -{u^2_*}/{4\e},
\end{equation*}
implying that all the eigenvalues $\lambda_k^\e$, $k \geq 2$ are bounded away from zero, as required in assumption {\bf H2}.

\begin{remark}\label{remarkstm}{ \rm

We notice that the size of the first eigenvalue with respect to $\e$ strongly depends on the choice of the boundary values, see \eqref{rough}. 
Indeed, if choosing $u_*=\e$ rather than $\sqrt{\e}$, estimate \eqref{rough} tells us that $\lambda_1^\e\approx\e$ as $\e\to 0$, 
that is, no metastable behavior is observed. 
This is confirmed also by numerical simulations, as we show in Figure \ref{fig7}.
This behavior  is consistent with the stability Theorem \ref{propstabilita1} we proved in Section 3 and with the subsequent comments; 
indeed, if choosing boundary values satisfying \eqref{ipodatibordo2} we can apply Theorem \ref{propstabilita1} and we thus have a {\it fast convergence} towards the equilibrium 
(as confirmed by numerical simulations and by the asymptotic expression of the first eigenvalue \eqref{rough}). 
Conversely, when $\vert u_\pm \vert=\sqrt{\e}$ (hence bigger), we observe a metastable behavior and clearly Theorem \ref{propstabilita1} does not hold.}
\end{remark}

\begin{figure}[ht]
\centering
\includegraphics[width=7cm,height=5.5cm]{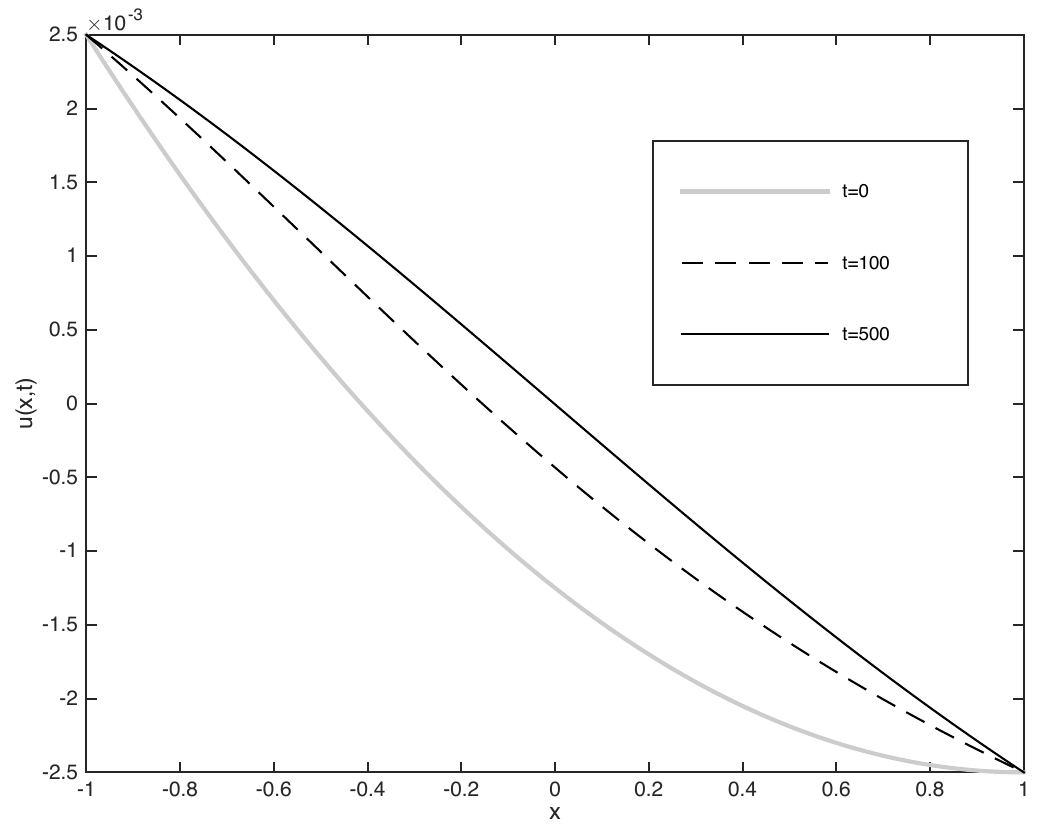}
 \caption{\small{The dynamics of the solution to \eqref{burgers2} for $\e=0.005$, $f(u)=u^2/2$ and $u_*=\e/2$. In this case, the equilibrium is attained in short times. }}\label{fig7}
 \end{figure}

\subsection{Conclusions}
Having proved that assumptions {\bf H1-H2} are satisfied, we now proceed similarly as in \cite{MS, Str17}, to which we refer for detailed computations,
ending up with a coupled system for the variables $(\xi, v)$ that reads
\begin{equation}\label{Svxi}
	\left\{\begin{aligned}
\frac{d\xi}{dt}&=\theta^\varepsilon(\xi)\bigl(1+\langle\partial_{\xi} \psi^\varepsilon_1, v \rangle\bigr)\\
\partial_t v& =  H^\e(x;\xi)+ \mathcal L ^\e v +\mathcal M^\e v,
\end{aligned}\right.
\end{equation}
where
\begin{align*}
\theta^\e(\xi)&:= \langle \psi^\varepsilon_1,\mathcal P^\e[U^{\varepsilon}] \rangle,\\
		H^\varepsilon(\cdot;\xi)&:={\mathcal P}^\varepsilon[U^{\varepsilon}(\cdot;\xi)]
			-\partial_{\xi}U^{\varepsilon}(\cdot;\xi)\,\theta^\varepsilon(\xi),\\
		{\mathcal M}^\varepsilon  v&:=-\partial_{\xi}U^{\varepsilon}(\cdot;\xi)
			\,\theta^\varepsilon(\xi)\,\langle\partial_{\xi} \psi^\varepsilon_1, v \rangle,
\end{align*}
and $\psi^\varepsilon_k$ is the $k$-th eigenfunction of the adjoint operator $\mathcal L^{\e,*}$.
The following theorem provides an estimate for the perturbation $v$ to be used in the ODE for the variable $\xi$ to decouple the system.
\begin{theorem}\label{teo56}
Let $v$ be the solution to \eqref{Svxi}  and let assumptions {\bf H1-H2} be satisfied. 
Then there exist $c,T > 0$ such that, for all $t<T$, it holds that
\begin{equation}\label{stimav}
\|v\|_{L^2}(t) \leq  \|v_0\|_{{}_{L^2}} e^{\nu^\e t}+ c \,  |\Omega^\e|_{{}_{{}_{L^\infty}}}t, \quad \nu^\e := c|\Omega^\e|_{{}_{L^\infty}}-\sup_{\xi} |\lambda_1^\e(\xi)|.
\end{equation}
\end{theorem}
\begin{proof}
The proof relies on the semigroup theory for linear operators depending on time developed in \cite{Pazy}, and it is an adaptation of the  one of \cite[Theorem 3.4]{Str17}; 
we report here the major modifications of the argument, referring the reader to the discussion and the definitions given in \cite[Section 3.1]{Str17}. 

First of all, by their very definitions we can state that $\mathcal M^\varepsilon$ is a bounded operator that satisfies the estimate
\begin{equation*}
	\|\mathcal M^\varepsilon\|_{\mathcal L(L^2;\R)} \leq c \, \|\theta^\varepsilon\|_{{}_{L^\infty}} \leq c \,  |\Omega^\varepsilon|_{{}_{L^\infty}}, 
\end{equation*}
and $H^\varepsilon$ is such that
\begin{equation}\label{asyMH2}
	\|H^\varepsilon(\cdot;\xi)\|_{{}_{L^\infty}} \leq c \, | \Omega^\varepsilon|_{{}_{L^\infty}}.
\end{equation}
Next, concerning the linear operator $\mathcal L^\e$, if we define $\Lambda_1^\varepsilon:=\sup_\xi \lambda_1^\varepsilon(\xi) <0$ we have 
$\sigma(\mathcal L^\e)\subset(-\infty,-|\Lambda_1^\varepsilon|)$ and $\mathcal L^\e$ is the infinitesimal generator of a $C^0$ semigroup $\mathcal S_{\xi(t)}(s)$, $s>0$, such that
\begin{equation}\label{stimaPazy}
	\| \mathcal S_{\xi(t)}(s)\| \leq e^{-|\Lambda_1^\varepsilon|s}, \qquad \forall \ t \geq 0.
\end{equation}
Hence, for any $\xi(t) \in I$,  the family $\{ \mathcal L^\varepsilon \}_{{}_{\xi(t) \in I}}$ is stable with stability constant $-|\Lambda_1^\varepsilon|$, 
implying that the family $\{ \mathcal L^\varepsilon + \mathcal M^\varepsilon \}_{{}_{\xi(t) \in I}}$ is stable with stability constant $ -|\Lambda_1^\varepsilon|+{ c \, }|\Omega^\varepsilon|_{{}_{L^\infty}}$.

We can now define $\mathcal T(t,s)$ as the {\it evolution system} of the linear equation $\partial_t w=(\mathcal L^\varepsilon+ \mathcal M^\varepsilon)w$,
so that we can rewrite the solution to \eqref{Svxi} as
\begin{equation}\label{Yfamiglieevol}
	v(t)=\mathcal T(t,s)v_0+\int_s^t \mathcal T(t,s)H^\varepsilon(x;\xi(r)) dr, \quad 0 \leq s \leq t  
\end{equation} 
and, because of \eqref{stimaPazy}, it holds
\begin{equation*}
	\|\mathcal T(t,s)\| \leq c \, e^{\nu^\varepsilon (t-s)}, \qquad \nu^\varepsilon :={c \, }|\Omega^\varepsilon|_{{}_{L^\infty}}- |\Lambda^\varepsilon_1|.
\end{equation*}
Hence, from the representation formula \eqref{Yfamiglieevol} it follows that
\begin{equation*}
	\|v\|_{{}_{L^2}}(t) \leq e^{\nu^\varepsilon t}\|v_0\|_{{}_{L^2}}+ \sup_{\xi \in I}\|H^\varepsilon(\cdot;\xi)\|_{{}_{L^\infty}}\int_0^te^{\nu^\varepsilon(t-s)} \ ds, \quad \forall \, t \geq 0,
\end{equation*}
from which, using \eqref{asyMH2}, we end up with \eqref{stimav} and the proof is complete.
\end{proof}

\begin{remark}{\rm
For the precise statements of the theorems we used in the proof of Theorem \ref{teo56}, we refer the reader to \cite[Section 5]{Pazy} and, in particular, to Definition 2.1, Theorem 2.3, Theorem 3.1 and Theorem 4.2 therein. 
}
\end{remark}
Since the terms $|\Omega^\e|_{{}_{L^\infty}}$ and $\sup_{\xi} |\lambda_1^\e(\xi)|$ are exponentially small as $\e\to0^+$, the bound obtained in \eqref{stimav} shows that $v$ is very small for large times and small $\e$, provided that $\|v_0\|_{{}_{L^2}}$ is small enough; 
such estimate can be used in the ODE for the variable $\xi(t)$, leading to
\begin{equation*}
	\frac{d\xi}{dt} = \theta^\e(\xi)(1+r), \quad \mbox{with} \quad |r|\leq C \left( \|v_0\|_{{}_{L^2}}e^{\nu^\e t} + t|\Omega^\e|_{{}_{L^\infty}}\right).
\end{equation*}
As $r$ is exponentially small as $\e\to0^+$, the motion of the interface location $\xi$ is described by the ODE
\begin{equation}\label{eq:xi}
	\frac{d\xi}{dt} = \theta^\e(\xi).
\end{equation}
To describe the properties of the solutions to the equation \eqref{eq:xi} with initial datum $\xi(0)=\xi_0$, we briefly discuss the properties of the function $\theta^\e$.
By the definition
\begin{equation*}
	\theta^\e(\xi)=\langle \psi^\varepsilon_1,{\mathcal P^\e[U^{\varepsilon}] \rangle}= \psi^\varepsilon_1(\xi)\left(\k_-(\xi)-\k_+(\xi)\right).
\end{equation*}
We thus need an asymptotic expression for the first eigenfunction $\psi_1^\e$ of $\mathcal L^{\e,*}$ as $\e\to0^+$; 
to this end, we again approximate $\psi_1^\e$ with the eigenfunction of $\mathcal N^{\e,*} $  (defined in \eqref{opiperbolico}) relative to the eigenvalue $\lambda_1^0=0$, that solves
\begin{equation*}
	\left\{\begin{aligned}
	&\d_x^2 \psi_1^0 \pm u_* \d_x \psi_1^0=0,		\\
	&\psi_1^0(\pm\ell)=0, \quad [ \![ \psi_1^0 ] \!]_{x=\xi}=0,
	\end{aligned}\right.
\end{equation*}
where $[ \![ \cdot ] \!]$ denotes the jump.     By solving such boundary value problem, we obtain
\begin{equation*}
	 \psi_1^0(x)=
		\left\{\begin{aligned}
			&C\,(1-e^{-u_*(\ell-\xi)/\varepsilon})(1-e^{-u_*(\ell+x)/\varepsilon}),
					&\qquad	&x<\xi,\\
			&C\,(1-e^{-u_*(\ell+\xi)/\varepsilon})(1-e^{-u_*(\ell-x)/\varepsilon}),
					&\qquad	&x>\xi,		
		\end{aligned}\right.
\end{equation*}
so that $\psi_1^0(\xi)\approx C>0$ as $\e \to 0^+$ for any $\xi\in(-\ell,\ell)$. 
Hence, the behavior of $\theta^\e$ as $\e \to 0$ is dictated by the one of the difference $g(\xi):=\k_-(\xi)-\k_+(\xi)$.
We have already studied the properties of the function $g$;
it is a monotone decreasing function and there exists a unique $\bar\xi\in(-\ell,\ell)$ such that $g(\bar\xi)=0$.
This implies that
\begin{equation*}
(\xi-\bar\xi)\theta^\e(\xi)<0 \quad \forall\,\xi\neq\bar\xi,  \qquad \mbox{ and } \qquad  \theta^\e(\bar\xi)=0.
\end{equation*}
Therefore, $\xi(t)$ is a monotone function which converges to the unique equilibrium position $\bar\xi$ as $t\to+\infty$ and, as a consequence, 
the interface moves towards the right (resp., left) if $\xi_0\! \in\! (-\ell,\bar\xi)$ (resp., $\xi_0 \! \in \! (\bar\xi, \ell)$). 
In terms of the original solution $u(x,t)=U^\e(x;\xi(t))+v(x,t)$ to \eqref{burgers2}, 
we have that the solution $u$ is converging to $U^\e(x;\bar\xi)$ for large times, being $U^\e(x;\bar\xi)$ the unique steady state for the system;
however, the speed rate of such convergence is dictated by the speed rate of the convergence of $\xi(t)$ towards $\bar \xi$ and so by the magnitude of $\theta^\e$.
As already seen, the difference $\k_-(\xi)-\k_+(\xi)$ is exponentially small as $\e\to0^+$ if the flux function $f$ and the boundary data are properly chosen. 
For instance, in the case of a Burgers flux $f(u)=u^2/2$, for which we recall that $u_\pm=\mp u_*$ and $\bar\xi=0$, one has
\begin{equation*}
|\theta^\e(\xi)|\leq c_1 u_*\,e^{-c_2u_*(\ell-|\xi|)/\e}.
\end{equation*}
In particular, if $u_*=\sqrt\e$, then
\begin{equation*}
|\theta^\e(\xi)|\leq c_1\sqrt\e\,e^{-c_2(\ell-|\xi|)/\sqrt\e},
\end{equation*}
and the velocity of the interface is exponentially small as $\e\to0^+$,
leading to a metastable behavior. 

On the contrary, if $ 0<u_*\leq \e$, then $\theta^\e$ (that is, the speed rate of convergence of the solutions towards the steady state) will be no longer exponentially small in $\e$: 
in particular, no metastability is observed in this case. 
This is consistent with the stability Theorem \ref{propstabilita1}, which states that in this case we have a fast convergence towards the equilibrium (see also Remark \ref{remarkstm}).

\smallbreak 
We conclude this paper by showing a numerical solution to the problem \eqref{burgers2} with $f(u)=u^2/2$ and 
a smooth initial datum $u_0$ which does not satisfy the smallness condition \eqref{ipodatibordo} (see Figure \ref{fig:discon}). 
In a relatively short time the numerical solution becomes discontinuous (as shown in the left picture of Figure \ref{fig:discon}, to be compared with \cite{GooKurRos}, where the same behavior has been observed);
after that it evolves very slowly and converges to the discontinuous steady state after a very long time.
Therefore, we still observe (at least numerically) a metastable behavior also in the case of discontinuous solutions.
 
\begin{figure}[ht]
\centering
\includegraphics[width=6.9cm,height=5.5cm]{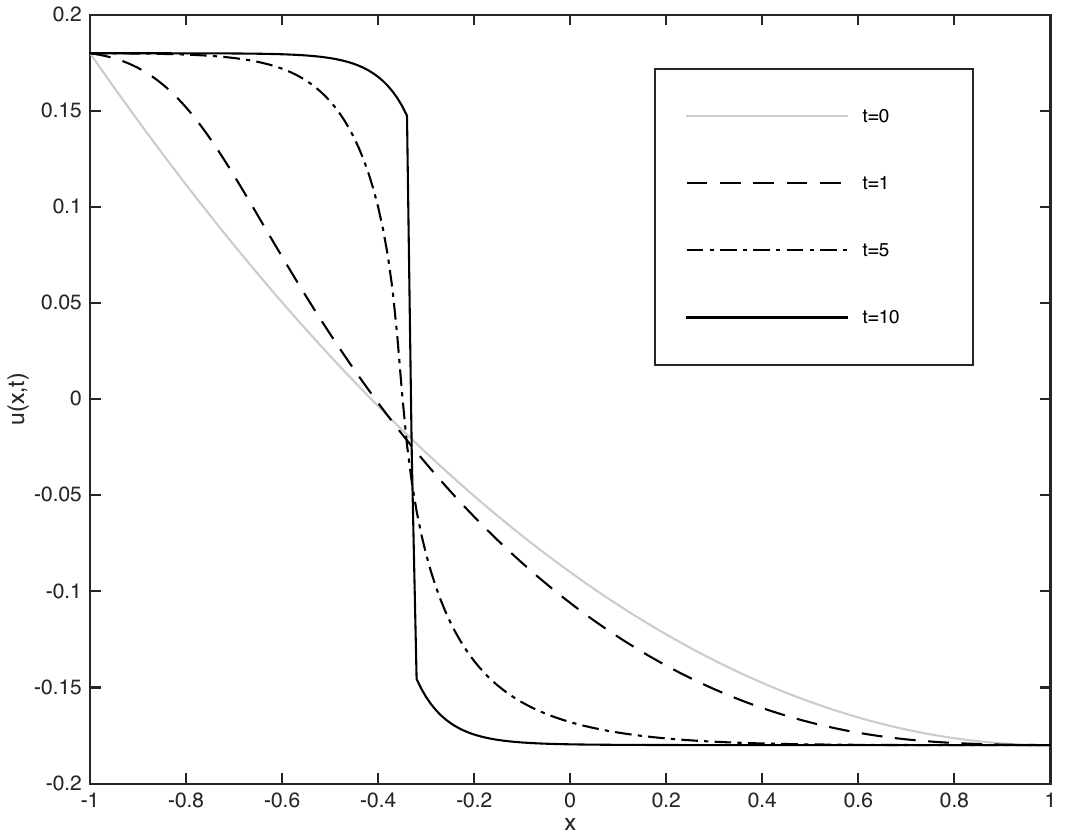}
 \hspace{3mm}
\includegraphics[width=7cm,height=5.45cm]{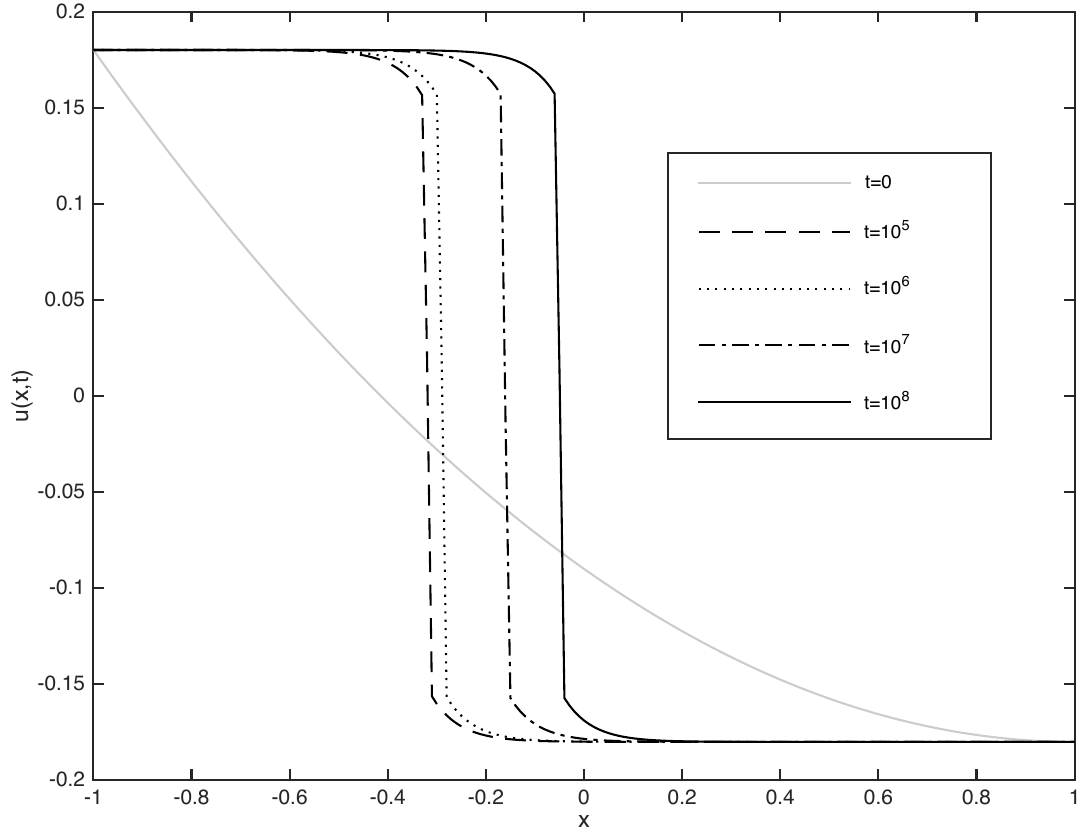}
 \hspace{3mm}
 \caption{\small{The dynamics of the solution to \eqref{burgers2} with $\e=0.01$, $u_\pm = \mp1.8 \sqrt{\e}$ and initial datum $u_0(x)=1.8\sqrt{\e}\left(\frac{1}{2}x^2-x-\frac{1}{2}\right)$.
 We observe a \emph{metastable} convergence of the solution to a discontinuous steady state.}}\label{fig:discon}
 \end{figure}

\section*{Acknowledgments} This is a pre-print of an article published in Journal of Evolution Equations. 
The final authenticated version is available online at: https://doi.org/10.1007/s00028-019-00528-2. 

We thank the anonymous referee for her/his comments that helped to improve the paper.

\end{document}